\documentclass[a4paper,11pt]{amsart}

\usepackage{amssymb,amsthm,amsmath,amscd,amsfonts,bbm,mathrsfs}
\usepackage{hyperref}

\usepackage[cmtip,all]{xy}

\theoremstyle{plain}
\newtheorem{Lemma}{Lemma}
\newtheorem{Thm}[Lemma]{Theorem}
\newtheorem*{Thm*}{Theorem}
\newtheorem{Prop}[Lemma]{Proposition}
\newtheorem{Cor}[Lemma]{Corollary}
\theoremstyle{definition}
\newtheorem{Defn}[Lemma]{Definition}

\theoremstyle{remark}
\newtheorem{Remark}[Lemma]{Remark}
\newtheorem*{Remark*}{Remark}

\numberwithin{Lemma}{section}
\numberwithin{equation}{section}

\newcommand{\AAA}{\mathcal{A}}
\newcommand{\BBB}{\mathcal{B}}
\newcommand{\CCC}{\mathcal{C}}
\newcommand{\DDD}{\mathcal{D}}
\newcommand{\FFF}{\mathcal{F}}
\newcommand{\GGG}{\mathcal{G}}

\newcommand{\MMM}{\mathcal{M}}
\newcommand{\OOO}{\mathcal{O}}

\newcommand{\RRR}{\mathcal{R}}

\newcommand{\Fa}{\mathfrak{a}}
\newcommand{\Fb}{\mathfrak{b}}
\newcommand{\Fm}{\mathfrak{m}}

\newcommand{\FM}{\mathfrak{M}}
\newcommand{\FN}{\mathfrak{N}}
\newcommand{\FS}{\mathfrak{S}}

\newcommand{\DD}{{\mathbb{D}}}
\newcommand{\FF}{{\mathbb{F}}}

\newcommand{\QQ}{{\mathbb{Q}}}
\newcommand{\ZZ}{{\mathbb{Z}}}

\DeclareMathOperator{\Ann}{Ann}
\DeclareMathOperator{\Aug}{Aug}
\DeclareMathOperator{\Aut}{Aut}
\DeclareMathOperator{\BK}{BK}
\DeclareMathOperator{\BT}{BT}
\DeclareMathOperator{\Coker}{Coker}

\DeclareMathOperator{\Def}{Def}

\DeclareMathOperator{\D}{D}
\DeclareMathOperator{\DF}{DF}
\DeclareMathOperator{\DM}{DM}
\DeclareMathOperator{\DDF}{\DD F}
\DeclareMathOperator{\End}{End}

\DeclareMathOperator{\Filone}{Fil}
\DeclareMathOperator{\tFilone}{\tilde{Fil}}

\DeclareMathOperator{\pGr}{{\it p}Gr}
\DeclareMathOperator{\Hom}{Hom}
\DeclareMathOperator{\Image}{Im}

\DeclareMathOperator{\Ker}{Ker}
\DeclareMathOperator{\Lie}{Lie}
\DeclareMathOperator{\LF}{LF}

\DeclareMathOperator{\Nil}{Nil}
\DeclareMathOperator{\Rad}{Rad}
\DeclareMathOperator{\Spec}{Spec}

\DeclareMathOperator{\Set}{Set}
\DeclareMathOperator{\Tor}{Tor}
\DeclareMathOperator{\T}{T}
\DeclareMathOperator{\Win}{Win}

\DeclareMathOperator{\bal}{bal}
\DeclareMathOperator{\cris}{cris}

\DeclareMathOperator{\et}{et}

\DeclareMathOperator{\gr}{gr}

\DeclareMathOperator{\id}{id}

\DeclareMathOperator{\per}{per}

\DeclareMathOperator{\tor}{tor}

\renewcommand{\u}{\underline}

\newcommand{\pfd}{R}

\begin{document}

\title{Dieudonn\'e theory over semiperfect rings and perfectoid rings}
\author{Eike Lau}
\date{\today}

\begin{abstract}
The Dieudonn\'e crystal of a $p$-divisible group over a semiperfect ring
$R$ can be endowed with a window structure. 
If $R$ satisfies a boundedness condition, this construction gives an
equivalence of categories. As an application we obtain a classification
of $p$-divisible groups and commutative finite locally free $p$-group schemes 
over perfectoid rings by Breuil-Kisin-Fargues modules if $p\ge 3$.
\end{abstract}

\maketitle


\section{Introduction}

Let $p$ be a prime.
A semiperfect ring is an $\FF_p$-algebra $R$ such that the Frobenius endomorphism
$\phi_R:R\to R$ is surjective.
In the first part of this article
we study the classification of $p$-divisible groups over semiperfect rings
by Dieudonn\'e crystals and related objects.
This was initiated in \cite{Scholze-Weinstein}.
In the second part we draw conclusions for perfectoid rings.

\subsection*{Crystalline Dieudonn\'e windows.}
Every semiperfect ring $R$ has a universal $p$-adic divided power extension 
$A_{\cris}(R)$.
By a lemma of \cite{Scholze-Weinstein}, 
this ring carries a natural structure of a frame $\u A{}_{\cris}(R)$, 
which means that the Frobenius of $A_{\cris}(R)$ is divided by $p$
on the kernel of $A_{\cris}(R)\to R$.
This is not clear a priori because in general $A_{\cris}(R)$ has $p$-torsion.

The following result has been suggested in \cite{Scholze-Weinstein}.

\begin{Thm}
\label{Th:Intro-PhiR}
For all semiperfect rings $R$ there are natural functors 
\[
\Phi_R^{\cris}:\BT(\Spec R)\to\Win(\u A{}_{\cris}(R))
\]
from $p$-divisible groups over $R$ to windows over $\u A{}_{\cris}(R)$
which endow the Dieudonn\'e crystal of a $p$-divisible group with a window structure.
\end{Thm}

See Theorem~\ref{Th:PhiR}. The functor $\Phi_R^{\cris}$ 
is a variant of the functor $\Phi_R$ of \cite{Lau:Smoothness} 
from $p$-divisible groups to displays for an arbitrary $p$-adic ring $R$, 
and of the functor $\Phi_R$ of \cite{Lau:Relation} from $p$-divisible groups 
to Dieudonn\'e displays for a local Artin ring $R$ with perfect residue field.

Our main result on the functor $\Phi_R^{\cris}$ depends on the following 
boundedness condition: We call $R$ balanced if $\Ker(\phi_R)^p=0$,
and we call $R$ iso-balanced if there is 
a nilpotent ideal $\Fa\subseteq R$ such that $R/\Fa$ is balanced.
Every $f$-semiperfect ring in the sense of \cite{Scholze-Weinstein} is iso-balanced.

\begin{Thm}
\label{Th:Intro-PhiR-iso-bal}
If $R$ is iso-balanced, the functor $\Phi_R^{\cris}$ is an equivalence.
\end{Thm}

See Theorem~\ref{Th:PhiR-iso-bal}.
Theorem~\ref{Th:Intro-PhiR-iso-bal} implies that for iso-balanced semiperfect rings
the crystalline Dieudonn\'e functor
\[
\DD_R:\BT(\Spec R)\to(\text{Dieudonn\'e crystals over }\Spec R)
\]
is fully faithful up to isogeny. 
When $R$ is $f$-semiperfect, this is proved in \cite{Scholze-Weinstein} using perfectoid spaces.

Assume that $R$ is a complete intersection in the sense that
$R$ is the quotient of a perfect ring by a regular sequence.
Then $A_{\cris}(R)$ is torsion free, and windows over $\u A{}_{\cris}(R)$
are equivalent to Dieudonn\'e crystals over $\Spec R$ with an admissible
filtration in the sense of \cite{Grothendieck:Montreal}; this filtration is unique
if it exists. 
Thus for complete intersections, 
Theorem~\ref{Th:Intro-PhiR-iso-bal} means that the functor $\DD_R$
is fully faithful and that its essential image consists of those Dieudonn\'e crystals 
which admit an admissible filtration.
Full faithfulness is already proved in \cite{Scholze-Weinstein}
as an easy consequence of full faithulness up to isogeny.

For a general semiperfect ring, $A_{\cris}(R)$ can have $p$-torsion, 
and the functor $\DD_R$ cannot be expected to be fully faithful.
The phenomenon that passing from Dieudonn\'e modules to windows 
can compensate this failure
is familiar from the classification of formal $p$-divisible groups over arbitrary
$p$-adic rings by nilpotent displays, and from the classification of arbitrary $p$-divisible
groups over local Artin rings by Dieudonn\'e displays.

\subsection*{Dieudonn\'e modules via lifts.}
The proof of Theorem~\ref{Th:Intro-PhiR-iso-bal} relies on another construction of 
Dieudonn\'e modules, which is independent of the functors $\Phi_R^{\cris}$.
A lift of an $\FF_p$-algebra $R$ is a $p$-adically complete and torsion free ring $A$ 
with $A/pA=R$ and with a Frobenius lift $\sigma:A\to A$. 
Then there is an evident frame structure $\u A$ and a functor
\[
\Phi_{A}:\BT(\Spec R)\to\Win(\u A).
\]
Here $\u A$-windows are equivalent to locally free Dieudonn\'e modules over $A$
in the usual sense. 
The functor $\Phi_A$ also induces a functor $\Phi_A^{\tor}$ from
commutative finite locally free $p$-group schemes over $R$ to $p$-torsion Dieudonn\'e
modules over $A$ which are of projective dimension $\le 1$ as $A$-modules.
In general the properties of $\Phi_A$ depend on the lift.

\begin{Thm}
\label{Th:Intro-PhiA}
If $R$ is a complete intersection or balanced semiperfect ring, 
there is a lift $A$ of $R$ such that the 
functors $\Phi_A$ and $\Phi_A^{\tor}$ are equivalences.
\end{Thm}

See Theorem~\ref{Th:PhiA-equiv} and Corollary~\ref{Co:PhiA-tor-equiv}.
When $R$ is perfect, then $A=W(R)$ is the unique lift of $R$, 
and Theorem~\ref{Th:Intro-PhiA} holds by a result of Gabber. 
The general case is reduced to the perfect case 
by a specialisation argument along $R^\flat\to R$,
where $R^\flat$ is the limit perfection of $R$.

We note that for an arbitrary $\FF_p$-algebra $R$ with a lift $(A,\sigma)$ the functor
$\Phi_A$ gives an equivalence between formal $p$-divisible groups and nilpotent windows
by \cite{Zink:Windows} and the extensions of \cite{Zink:Display} provided by 
\cite{Lau:Disp, Lau:Smoothness}. So the new aspect of Theorem~\ref{Th:Intro-PhiA}
is that it applies to all $p$-divisible groups.

\medskip
The functors $\Phi_R^{\cris}$ and $\Phi_A$ are related as follows.
For every lift $A$ of a semiperfect ring $R$ there is a natural homomorphism of frames 
\[
\varkappa:\u A{}_{\cris}(R)\to \u A,
\]
and the base change under $\varkappa$ of $\Phi_R^{\cris}(G)$ coincides with $\Phi_A(G)$.

\begin{Lemma}
\label{Le:Intro}
If $R$ is a complete intersection or balanced semiperfect ring, 
there is a lift $A$ of $R$ as in Theorem~\ref{Th:Intro-PhiA} such that
$\varkappa$ induces an equivalence of the window categories.
\end{Lemma}

See Proposition~\ref{Pr:str-kappa-cryst}.
Theorem~\ref{Th:Intro-PhiA} and Lemma~\ref{Le:Intro} give
Theorem~ \ref{Th:Intro-PhiR-iso-bal} when $R$ is balanced or a complete intersection,
and the general case follows by a deformation argument, 
using a weak version of lifts for iso-balanced rings, 
for which an analogue of Lemma~\ref{Le:Intro} holds; 
see Proposition~\ref{Pr:isobal-kappa-cryst}.

\subsection*{Breuil-Kisin-Fargues modules.}

Now let $\pfd$ be a perfectoid ring in the sense of \cite{BMS}.
This class of rings includes all perfect rings and all 
bounded open integrally closed subrings of perfectoid Tate rings
in the sense of \cite{Fontaine:Perfectoides}.
Let $\pfd^\flat$ be the tilt of $\pfd$, which is a perfect ring, 
and $A_{\inf}=W(\pfd^\flat)$.

The kernel of the natural homomorphism $\theta:A_{\inf}\to \pfd$ is generated
by a non-zero divisor $\xi$. In the following,
a Breuil-Kisin-Fargues module for $\pfd$ is a finite projective $A_{\inf}$-module
$\FM$ with a linear map $\varphi:\FM^\sigma\to\FM$ whose cokernel
is annihilated by $\xi$.
As an application of Theorem~\ref{Th:Intro-PhiR-iso-bal} we obtain:

\begin{Thm}
\label{Th:Intro-Perfectoid}
If $p\ge 3$, for each perfectoid ring $\pfd$
the category $\BT(\Spec \pfd)$ is equivalent to the category of 
Breuil-Kisin-Fargues modules for $\pfd$.
\end{Thm}

See Theorem~\ref{Th:BT-DM}.
When $\pfd=\OOO_C$ for an algebraically closed perfectoid field $C$,
the result is due to Fargues \cite{Fargues:AuDela,Fargues:Paderborn}.
Theorem~\ref{Th:Intro-Perfectoid} is a variant of the classical equivalence 
between $p$-divisible groups over a mixed characteristic
complete discrete valuation ring with perfect residue field and Breuil-Kisin modules.

To prove Theorem~\ref{Th:Intro-Perfectoid} 
we consider the balanced semiperfect ring  $\pfd/p$. 
The universal $p$-adic divided power extension $A_{\cris}(\pfd)$
coincides with $A_{\cris}(\pfd/p)$ as a ring and carries a natural frame structure.
The equivalence of Theorem~\ref{Th:Intro-PhiR-iso-bal} for $\pfd/p$ 
(which is covered by Theorem~\ref{Th:Intro-PhiA} and Lemma~\ref{Le:Intro} in this case)
extends for $p\ge 3$ to an equivalence
\[
\BT(\Spec \pfd)\to\Win(\u A{}_{\cris}(\pfd)).
\]
Moreover there is a base change functor 
\[
(\text{Breuil-Kisin-Fargues modules for $\pfd$})\to\Win(\u A{}_{\cris}(\pfd)),
\]
which is an equivalence for $p\ge 3$ by a descent from $A_{\cris}$ to $A_{\inf}$ that
generalizes the `descent from $S$ to $\FS$' used in the classical case.
Theorem~\ref{Th:Intro-Perfectoid} follows. 
One can expect that Theorem~\ref{Th:Intro-Perfectoid} also holds for $p=2$,
but the present proof does not extend to that case directly.

As in the classical case, Theorem~\ref{Th:Intro-Perfectoid} induces a
similar result for finite group schemes. 
Namely, a torsion Breuil-Kisin-Fargues module for $\pfd$ is a triple
$(\FM,\varphi,\psi)$ where $\FM$ is a $p$-torsion finitely presented $A_{\inf}$-module
of projective dimension $\le 1$ with linear maps
\[
\xi A_{\inf}\otimes_{A_{\inf}}\FM\xrightarrow{\;\psi\;} \FM^\sigma\xrightarrow{\;\varphi\;}\FM
\]
such that $\varphi\circ\psi$ and $\psi\circ(1\otimes\varphi)$ 
are the multiplication maps. 
If $\pfd$ is torsion free then $\xi$ is $\FM$-regular and $\psi$ is dermined by $\varphi$.

\begin{Cor}
(Theorem~\ref{Th:pGr-DM})
If $p\ge 3$, for each perfectoid ring $\pfd$
the category of commutative finite locally free $p$-group schemes over $\pfd$
is equivalent to the category of torsion Breuil-Kisin-Fargues modules for $\pfd$.
\end{Cor}
\subsection*{Acknowledgements}
The author thanks Peter Scholze and Thomas Zink for helpful discussions.

\setcounter{tocdepth}{1}
\tableofcontents

\section{Notation}
\label{Se:Not}

We fix a prime $p$.

 A PD extension is a surjective ring homomorphism
whose kernel is equipped with divided powers.
A $p$-adic PD extension is a PD extension of $p$-adically complete rings
such that the divided powers are compatible with the divided powers on $p\ZZ_p$.
Divided powers $\gamma$ are also denoted by $\gamma_n(x)=x^{[n]}$.

Following \cite{Lau:Frames}, a frame $\u S=(S,\Filone S,R,\sigma,\sigma_1)$ 
consists of rings $S$ and $R=S/\Filone S$ such that $pS+\Filone S\subseteq\Rad S$, 
together with a Frobenius lift $\sigma:S\to S$ and a $\sigma$-linear map $\sigma_1:\Filone S\to S$ 
whose image generates the unit ideal.\footnote{Form a systematic point of view, 
it would be better to drop this condition; see for example \cite[\S 2.1]{Cais-Lau}. 
The condition is satisfied for all frames considered in this article.}
We denote by $\Win(\u S)$ the category of windows over $\u S$ 
in the sense of \cite{Lau:Frames}.
A frame homomorphism $\alpha:\u S\to\u S'$ is a ring homomorphism $S\to S'$
with $\Filone S\to\Filone S'$ such that $\sigma'\alpha=\alpha\sigma$ and
$\sigma'_1\alpha=u\cdot \alpha\sigma_1$ for a unit $u\in S'$.
There is a base change functor $\alpha^*:\Win(\u S)\to\Win(\u S')$.
If this functor is an equivalence, $\alpha$ is called crystalline.

A frame $\u S$ is called a $p$-frame if $p\sigma_1=\sigma$ on $\Filone S$.
A PD frame is a $p$-frame $\u S$ where $S\to R$ is a $p$-adic PD extension 
such that $\sigma$ preserves the resulting divided powers on the ideal $\Filone S+pS$. 
If in addition $S$ is torsion free, then $(S,\sigma)$
is a frame for $R$ in the sense of \cite{Zink:Windows}.


\section{Dieudonn\'e crystals and modules}
\label{Se:DCM}

In this section we fix notation and recall some standard results.

For a scheme $X$ on which $p$ is nilpotent, or more generally a $p$-adic formal scheme,
let $\BT(X)$ be the category of $p$-divisible groups over $X$,
let $\D(X)$ be the category of locally free Dieudonn\'e crystals over $X$,
and let $\DF(X)$ be the category of locally free Dieudonn\'e crystals $\MMM$ over $X$ 
equipped with an admissible filtration $\Filone\MMM_X\subseteq\MMM_X$
as in \cite{Grothendieck:Montreal};
see \cite[Def.\ 2.4.1]{Cais-Lau}. Let
\begin{equation}
\label{Eq:DX}
\DD_X:\BT(X)\to\D(X)
\end{equation}
be the contravariant crystalline Dieudonn\'e functor defined in \cite{Mazur-Messing} and in \cite{BBM}, 
and let
\begin{equation}
\label{Eq:DFX}
\DDF_X:\BT(X)\to\DF(X)
\end{equation}
be its extension defined by the Hodge filtration; see \cite[Prop.\ 2.4.3]{Cais-Lau}.
If $\u S=(S,\Filone S,R,\sigma,\sigma_1)$ is a torsion free PD frame as in \S\ref{Se:Not},
the evaluation of the filtered Dieudonn\'e crystal at $\u S$ gives a contravariant functor
\begin{equation}
\label{Eq:PhiS}
\Phi_{\u S}:\BT(\Spec R)\to\Win(\u S),\qquad
G\mapsto(M,\Filone M,F,F_1)
\end{equation}
where $M=\DD(G)_S$, the submodule $\Filone M\subseteq M$ is the inverse image 
of the Hodge filtration $\Lie(G)^*\subseteq\DD(G)_R$ of $G$,
$F$ is induced by the Frobenius of $G$, and  $F_1=p^{-1}F$ on $\Filone M$;
see \cite[Prop.\ 3.17]{Lau:Relation} or \cite[Prop.\ 2.5.2]{Cais-Lau}.

\subsection*{Explicit Dieudonn\'e modules}

Let $R$ be an $\FF_p$-algebra.
A lift of $R$ is a pair $(A,\sigma)$ where $A$ is a $p$-adically complete and torsion free ring 
with $R=A/pA$, and $\sigma:A\to A$ is  a Frobenius lift.

In the following let $(A,\sigma)$ be a lift of $R$.
A (locally free) Dieudonn\'e module over $A$ 
is a triple $\u M=(M,\varphi,\psi)$ where $M$ a finite projective $A$-module
and $\varphi:M^\sigma\to M$ and $\psi:M\to M^\sigma$ 
are linear maps with $\varphi\psi=p$ and $\psi\varphi=p$,
where $M^\sigma=M\otimes_{A,\sigma}A$.
We write $\DM(A)$ for the category of Dieudonn\'e modules over $A$.

\begin{Lemma}
\label{Le:DMA-proj}
For $(M,\varphi,\psi)\in\DM(A)$ the $R$-module $\Coker(\varphi)$ is projective.
\end{Lemma}

\begin{proof} 
Let $\bar M=M\otimes_AR$. 
There is an exact sequence of finite projective $R$-modules
\begin{equation}
\label{Eq:DM-proj}
\bar M^\sigma\xrightarrow{\bar \varphi}\bar M\xrightarrow {\bar\psi}
\bar M^\sigma\xrightarrow {\bar\varphi}\bar M,
\end{equation}
and we have to show that $\Image(\bar\psi)$ is a direct summand of $\bar M^\sigma$.
This holds if and only if for each maximal ideal $\Fm\subset R$ 
the base change of \eqref{Eq:DM-proj} to $k=R/\Fm$ is exact, 
or equivalently if the base change to $k^{\per}$ is exact.

Let $\Delta:A\to W(A)$ be the homomorphism with $w_n\circ\Delta=\sigma^n$, 
where $w_n$ is the $n$-the Witt polynomial; see \cite[IX, \S1.2, Prop.\ 2]{Bourbaki:Comm-Alg}. 
The composition of $\Delta$ with $W(A)\to W(R)\to W(k^{\per})$
is a homomorphism $A\to W(k^{\per})$ that commutes with $\sigma$.
Then $M\otimes_AW(k^{\per})$ is a Dieudonn\'e module whose
reduction mod $p$ is $\eqref{Eq:DM-proj}\otimes_Rk^{\per}$, 
which is therefore exact as required.
\end{proof}

We have a frame $\u A=(A,pA,R,\sigma,\sigma_1)$ with $\sigma_1(pa)=\sigma(a)$,
and $\u A$ is a torsion free PD frame as defined in \S\ref{Se:Not}.
Using Lemma~\ref{Le:DMA-proj} one verifies that there is an equivalence of categories
\begin{equation}
\label{Eq:Win-DMA}
\Win(\u A)\to\DM(A), \qquad (M,\Filone M,F,F_1)\mapsto(N,\varphi,\psi)
\end{equation}
defined by $N=\Filone M$ and $\varphi(x\otimes 1)=F(x)$ for $x\in\Filone M$;
see \cite[Lemma~2.1.15]{Cais-Lau} with $E=p$. 
Thus the functor $\Phi_{\u S}$ of \eqref{Eq:PhiS} for $\u S=\u A$ can be viewed as a 
contravariant functor
\begin{equation}
\label{Eq:BT-DM}
\Phi_{A}:\BT(\Spec R)\to\DM(A).
\end{equation}
In certain cases one can hope that $\Phi_A$ is an equivalence of categories;
see \cite{Jong:Dieud} for the case of complete regular local rings.

\begin{Remark}
The functor $\Phi_A$ always induces an equivalence between formal $p$-divisible groups
and $\varphi$-nilpotent Dieudonn\'e modules,
which correspond to $F$-nilpotent $\u A$-windows.
This follows from \cite{Zink:Windows} together with the extension of
\cite[Thm.\ 9]{Zink:Display} to general base rings in \cite{Lau:Disp, Lau:Smoothness}.
\end{Remark}


\section{Semiperfect rings}
\label{Se:Semiperfect}

Let $p$ be a prime.
Following \cite{Scholze-Weinstein}, an $\FF_p$-algebra $R$ is called semiperfect 
if the Frobenius endomorphism $\phi:R\to R$ is surjective. 
An isogeny of semiperfect rings is a surjective ring homomorphism 
whose kernel is annihilated by a power of $\phi$.
Let $R$ be semiperfect.
There is a universal homomorphism 
\[
R^\flat\to R
\]
from a perfect ring to $R$, and there is a universal $p$-adic PD extension 
\[
A_{\cris}(R)\to R.
\]
Explicitly, we have $R^\flat=\varprojlim(R,\phi)$,
and $A_{\cris}(R)$ is the $p$-adic completion of the PD envelope 
of the natural map $W(R^\flat)\to R$.
We will often write $J=\Ker(R^\flat\to R)$.
Two classes of semiperfect rings will play a special role:
complete intersections and balanced rings.

\subsection{Complete intersection semiperfect rings}

\begin{Defn}
\label{Def:ci}
A semiperfect ring $R$ is called a complete intersection if one can write
$R=R_0/J_0$ where $R_0$ is a perfect ring and where the ideal $J_0$ 
is generated by a regular sequence.
\end{Defn}

\begin{Lemma}
\label{Le:ci}
Let $R=R_0/J_0$ as in Definition~\ref{Def:ci}
where $J_0$ is generated by the regular sequence $\u u=(u_1,\ldots, u_r)$.
The natural homomorphism $R_0\to R^\flat$ maps $\u u$ to a regular
sequence that generates the kernel of $R^\flat\to R$. 
\end{Lemma}

\begin{proof}
Since the ideal $J_0$ is finitely generated, 
the $J_0$-adic topology of $R_0$ coincides with the linear topology defined by the ideals
$\phi^n(J_0)$ for $n\ge 0$. 
Thus $R^\flat$ is the $J_0$-adic completion of $R_0$.
Then the assertion is clear.
\end{proof}

\begin{Remark}
\label{Rk:Acris-ci}
Lemma~\ref{Le:ci} implies that in Definition~\ref{Def:ci} one can take $R_0=R^\flat$.
It follows that for a complete intersection semiperfect ring $R$ the ring $A_{\cris}(R)$ is torsion
free; see for example \cite[Lemma~2.6.1]{Cais-Lau}.
\end{Remark}

\subsection{Balanced semiperfect rings}

\begin{Defn}
\label{Def:bal}
A semiperfect ring $R$ is called 
balanced if the ideal $\bar J=\Ker(\phi:R\to R)$ satisfies $\bar J^p=0$,
and $R$ is called iso-balanced if $R$ is isogeneous to a balanced semiperfect ring.
\end{Defn}

\begin{Lemma}
\label{Le:ker-hom-bal}
For a homomorphism of semiperfect rings $\alpha:R'\to R$ where $R$ is balanced
we have $\Ker(\alpha)^p=\phi(\Ker(\alpha))$.
\end{Lemma}

\begin{proof}
Clearly $\phi(\Ker(\alpha))\subseteq\Ker(\alpha)^p$.
To prove the opposite inclusion, let $x_1,\ldots,x_p\in\Ker(\alpha)$ be given,
and choose $y_i\in R'$ with $\phi(y_i)=x_i$.
Then $\alpha(y_i)\in\bar J=\Ker(\phi:R\to R)$.
Since $R$ is balanced we have $\alpha(\prod y_i)=0$, 
thus $\prod x_i=\phi(\prod y_i)\in\phi(\Ker(\alpha))$ as required.
\end{proof}

\begin{Lemma}
\label{Le:bal-J}
A semiperfect ring $R$ is balanced iff the ideal 
$J=\Ker(R^\flat\to R)$ satisfies $J^p=\phi(J)$.
\end{Lemma}

\begin{proof}
If $R$ is balanced then $J^p=\phi(J)$ by Lemma~\ref{Le:ker-hom-bal}. The rest is clear.
\end{proof}

\begin{Remark}
For every semiperfect ring $R$ there is a universal homomorphism to a balanced
semiperfect ring $R\to R^{\bal}$, namely $R^{\bal}=R^\flat/J^{\bal}$ where $J^{\bal}$ 
is the ascending union of the ideals $\phi^{-n}(J)^{p^n}$ for $n\ge 0$.
The ring $R$ is iso-balanced iff $R\to R^{\bal}$ is an isogeny.
\end{Remark}

\begin{Lemma}
\label{Le:isog-ker-nil}
Let $\pi:R'\to R$ be an isogeny of iso-balanced semiperfect rings.
Then the ideal\/ $\Ker(\pi)$ is nilpotent.
\end{Lemma}

\begin{proof}
The composition $\alpha:R'\xrightarrow\pi R\to R^{\bal}$ is an isogeny since $R$ is iso-balanced.
Lemma~\ref{Le:ker-hom-bal} implies that 
$\Ker(\alpha)^{p^n}=\phi^n(\Ker(\alpha))$, which is zero for large $n$.
\end{proof}

\begin{Remark}
\label{Rk:f-semiperfect}
A semiperfect ring $R$ is called $f$-semiperfect 
(\cite[Def.\ 4.1.2]{Scholze-Weinstein})
if it is isogeneous to the quotient of a perfect ring by a finitely generated ideal.
Each $f$-semiperfect ring is iso-balanced.
\end{Remark}

\subsection{Lifts of semiperfect rings}

Let $R$ be a semiperfect ring,  and let $J=\Ker(R^\flat\to R)$.

\begin{Defn}
\label{Def:lift}
A lift of $R$
is a $p$-adically complete and torsion free ring $A$ with $A/pA=R$
which carries a ring endomorphism $\sigma:A\to A$ that induces $\phi$ on $R$.
\end{Defn}

\begin{Remark}
\label{Rk:lift}
The endomorphism $\sigma:A\to A$ is unique if it exists.
Indeed, the universal property of the ring of Witt vectors 
\cite[Chap.\ IV, Prop.\ 4.3]{Grothendieck:Montreal}
gives a unique homomorphism 
$\psi:W(R^\flat)\to A$ that induces the projection $R^\flat\to R$ modulo $p$,
and we have $\psi\circ\sigma=\sigma\circ\psi$ by the universal property.
Moreover $\psi$ is surjective, and the uniqueness of $\sigma$ follows.
This reasoning shows that lifts $A$ of $R$ correspond to closed ideals
$J'\subseteq W(R^\flat)$ such that $\sigma(J')\subseteq J'$ and $J'\cap pW(R^\flat)=pJ'$ 
and $J'/pJ'=J$.
\end{Remark}

\begin{Defn}
\label{Def:str-lift}
A lift $A$ of $R$ is called straight if $A=W(R^\flat)/J'$
such that the set of all $a\in J$ with $[a]\in J'$ generates $J$.
\end{Defn}

\begin{Lemma}
\label{Le:exist-lift}
Let $R$ be a semiperfect ring which is a complete intersection or balanced,
see Definitions \ref{Def:ci} and \ref{Def:bal}. Then a straight lift of $R$ exists.
\end{Lemma}

\begin{proof}
If $R$ is a complete intersection, $J$ is generated by a regular sequence $(u_1,\ldots,u_r)$;
see Lemma~\ref{Le:ci}.
Let $J'=([u_1],\ldots,[u_r])$ in $W(R^\flat)$.
The ring $A=W(R^\flat)/J'$ is $p$-adically complete and torsion free with $A/pA=R$.
The ideal $J'$ is stable under $\sigma$ since $\sigma([u_i])=[u_i]^p$. 
Thus $A$ is a straight lift of $R$.

Assume that $R$ is balanced.
Let $J'\subseteq W(R^\flat)$ be the set of all Witt vectors 
$a=(a_0,a_1,\ldots)$ with $a_i\in\phi^i(J)$.
We claim that $J'$ is an ideal.
Indeed, the ring structure of $W(R^\flat)$ is given by 
$(x_0,x_1,\ldots)*(y_0,y_1,\ldots)=(g_0^*(x,y),g_1^*(x,y),\ldots)$ where $*$ is $+$ or $\times$,
with certain polynomials $g_n^*$.
If the variables $x_i,y_i$ have degree $p^i$, then $p_n^+$ is homogeneous of degree $p^n$,
and $p_n^\times$ is bihomogeneous of bidegree $(p^n,p^n)$. 
Since $R$ is balanced we have $\phi(J)=J^p$; see Lemma~\ref{Le:bal-J}.
It follows that $J'$ is an ideal.
We have $J'\cap pW(R^\flat)=pJ'$,
and $J'$ is the closure of the ideal generated by the elements $[a]$ for all $a\in J$.
Clearly $J'$ is stable under $\sigma$.
Thus $A=W(R^\flat)/J'$ is a straight lift of $R$.
\end{proof}

\begin{Lemma}
\label{Le:lim-A-sigma}
If $A$ is a lift of the semiperfect ring $R$, then $\sigma:A\to A$ is surjective,
and 
\[
\varprojlim(A,\sigma)=W(R^\flat).
\]
\end{Lemma}

\begin{proof}
The first assertion holds because the natural $\sigma$-equivariant homomorphism
$W(R^\flat)\to A$ is surjective, and $\sigma$ is bijective on $W(R^\flat)$;
see Remark~\ref{Rk:lift}.
Let $B=\varprojlim(A,\sigma)$. Since $A$ is torsion free the same holds for $B$.
We take the limit over $\sigma$ of the exact sequence $0\to A\xrightarrow {} A\to A_n\to 0$, 
where the first map is $p^n$. 
It follows that $B/p^nB=\varprojlim(A_n,\sigma)$, which implies that $\varprojlim_n(B/p^n B)=B$,
moreover $B/pB=R^\flat$. Therefore $B=W(R^\flat)$.
\end{proof}


\section{Dieudonn\'e modules via lifts}
\label{Se:Dmod-via-lifts}

Let $R$ be a semiperfect ring and let $A$ be lift of $R$; see Definition~\ref{Def:lift}.

\subsection{Frames associated to a lift}

To the lift $A$ of $R$ we associate two frames.
First, there is the torsion free PD frame
\[
\u A=(A,pA,R,\sigma,\sigma_1)
\]
with $\sigma_1=p^{-1}\sigma$; see \S\ref{Se:DCM}.
Second, let
\[
\tFilone A=\Ker(A\to R\xrightarrow\phi R).
\]

\begin{Lemma}
\label{Le:tildeFilA}
We have $\sigma(\tFilone A)\subseteq pA$, and\/ $\tFilone A$ is a PD ideal of $A$.
\end{Lemma}

\begin{proof} 
Since $\sigma$ is a lift of $\phi$, for $a\in A$ we have 
$a\in\tFilone A$ iff $\sigma(a)\in pA$ iff $a^p\in pA$. 
For $a\in\tFilone A$ let $b=a^p/p\in A$. 
We have to show that $b\in\tFilone A$, or equivalently that $\sigma(b)\in pA$.
But $\sigma(b)=\sigma(a)^p/p=p^{p-1}(\sigma(a)/p)^p$.
\end{proof}

Since $R$ is semiperfect, $\sigma$ induces an isomorphism 
$A/\tFilone A\xrightarrow\sim R$. By Lemma~\ref{Le:tildeFilA}
we can define a torsion free PD frame
\[
\u A/\phi=(A,\tFilone A,R,\sigma,\sigma_1)
\]
with $\sigma_1=p^{-1}\sigma$. 
The endomorphism $\sigma$ of $A$ is a frame endomorphism $\sigma:\u A\to\u A$
over $\phi:R\to R$, which factors into frame homomorphisms
\begin{equation}
\label{Eq:iota-pi}
\u A\xrightarrow\iota\u A/\phi\xrightarrow\pi\u A
\end{equation}
where $\iota$ is given by the identity on $A$ and by $\phi$ on $R$, 
while $\pi$ is given by $\sigma$ on $A$ and by the identity on $R$.

\begin{Lemma}
\label{Le:sigma-crystalline}
The frame homomorphism $\pi:\u A/\phi\to\u A$ is crystalline, 
i.e.\ it induces an equivalence of the window categories.
\end{Lemma}

\begin{proof}
Let $I$ be the kernel of the surjective homomorphism $\sigma:A\to A$.
If we write $A=W(R^\flat)/J'$ (see Remark~\ref{Rk:lift}) then $I=J'/\sigma(J')$.
Thus $\sigma=p\sigma_1$ is zero on $I$.
Since $A$ is torsion free it follows that $\sigma_1:I\to I$ is zero,
and the lemma follows from the general deformation lemma \cite[Th.\ 3.2]{Lau:Frames}.
\end{proof}

\begin{Remark}
\label{Rk:PD-Ker-phi}
The divided powers on $\tFilone A$, which exist by Lemma~\ref{Le:tildeFilA},
induce divided powers on the ideal $(\tFilone A)/pA=\Ker(\phi:R\to R)$ of $R$. 
Thus the given lift $A$ of $R$ determines divided powers on $\Ker(\phi)$.
\end{Remark}

\begin{Lemma}
\label{Le:PD-nil}
If $A$ is a straight lift of $R$ in the sense of Definition~\ref{Def:str-lift},
then the associated divided powers on $\Ker(\phi)$ are pointwise nilpotent.
\end{Lemma}

\begin{proof}
Let $J=\Ker(R^\flat\to R)$ and $A=W(R^\flat)/J'$.
Since $A$ is straight, there are generators $a_i$ of $J$ with $[a_i]\in J'$. 
The elements $b_i=\phi^{-1}(a_i)+J$ of $R$  generate the ideal $\Ker(\phi)$.
We claim that $b_i^{[p]}=0$, which proves the lemma.
The element $c_i=[\phi^{-1}(a_i)]+J'$ of $A$ is an inverse image of $b_i$.
We have $c_i^p=[a_i]+J'=0$ in $A$, thus $c_i^{[p]}=0$ in $A$, and thus $b_i^{[p]}=0$ in $R$.
\end{proof}

\subsection{Evaluation of crystals}
\label{Se:ev-crys}

We consider the functor
\begin{equation}
\label{Eq:PhiA-semiperf}
\Phi_A:\BT(\Spec R)\to\Win(\u A)
\end{equation}
given by \eqref{Eq:PhiS} for $\u S=\u A$.
Here $\Win(\u A)$ is equivalent to the category $\DM(A)$ of Dieudonn\'e modules
over $A$ by \eqref{Eq:Win-DMA}.

\begin{Prop}
\label{Pr:Cartesian-sigma}
If the divided powers on $\Ker(\phi)$ 
given by Remark~\ref{Rk:PD-Ker-phi}
are pointwise nilpotent,
then the commutative diagram of categories
\begin{equation}
\label{Eq:cart-sigma}
\xymatrix@M+0.2em{
\BT(\Spec R)
\ar[r]^{\phi^*}
\ar[d]_{\Phi_A} &
\BT(\Spec R) 
\ar[d]^{\Phi_A} \\
\Win(\u A)
\ar[r]^{\sigma^*} &
\Win(\u A)
}
\end{equation}
is cartesian.
\end{Prop}

\begin{proof}
The diagram \eqref{Eq:cart-sigma} commutes by the functoriality of $\Phi_A$ 
with respect to the frame endomorphism $\sigma:\u A\to\u A$.
The factorisation \eqref{Eq:iota-pi} of $\sigma$ induces the following 
extension of \eqref{Eq:cart-sigma}.
\begin{equation}
\label{Eq:deform}
\xymatrix@M+0.2em{
\BT(\Spec R) \ar[r]^{\phi^*} \ar[d]_{\Phi_A} &
\BT(\Spec R) \ar[r]^{\id} \ar[d]_{\Phi_{\u A/\phi}} &
\BT(\Spec R)  \ar[d]^{\Phi_A} \\
\Win(\u A) \ar[r]^{\iota^*} &
\Win(\u A/\phi) \ar[r]^-{\pi^*}_-{\sim} &
\Win(\u A) 
}
\end{equation}
Here $\pi^*$ is an equivalence by Lemma~\ref{Le:sigma-crystalline}.
Thus \eqref{Eq:cart-sigma} is equivalent to the left hand square of \eqref{Eq:deform}.
For a $p$-divisible group $G$ over $R$ let $\u M=\Phi_{\u A/\phi}(G)$ in $\Win(\u A/\phi)$.
Then $M\otimes_AR$ is the value of $\DD(G)$ at the PD extension $\phi:R\to R$.
Thus lifts of the Hodge filtration of $G$ under $\phi$ correspond to lifts of the
Hodge filtration of $\u M$ under $\iota^*$. 
The latter correspond to lifts of $\u M$ under $\iota^*$ by \cite[Lemma~4.2]{Lau:Frames},
and the former correspond to lifts of $G$ under $\phi$ by the Grothendieck-Messing
Theorem \cite{Messing:Crystals} 
since the divided powers on $\Ker(\phi)$ are pointwise nilpotent.
\end{proof}

\begin{Cor}
\label{Co:Rflat-R-cart}
If the divided powers on $\Ker(\phi)$ given by Remark~\ref{Rk:PD-Ker-phi}
are pointwise nilpotent, then the commutative diagram of categories
\[
\xymatrix@M+0.2em{
\BT(\Spec R^\flat)
\ar[r]
\ar[d]_{\Phi_{W(R^\flat)}} &
\BT(\Spec R) 
\ar[d]^{\Phi_A} \\
\Win(\u W(R^\flat))
\ar[r]&
\Win(\u A)
}
\]
is cartesian.
\qed
\end{Cor}

\begin{proof}
Proposition~\ref{Pr:Cartesian-sigma} gives a cartesian diagram
\[
\xymatrix@M+0.2em{
\varprojlim(\BT(\Spec R),\phi^*) 
\ar[r]
\ar[d]_{\varprojlim\Phi_A} &
\BT(\Spec R) 
\ar[d]^{\Phi_A} \\
\varprojlim(\Win(\u A),\sigma^*)
\ar[r]&
\Win(\u A)
}
\]
The upper limit category is equivalent to $\BT(\Spec R^\flat)$ by the obvious
analogue of \cite[Ch.~II, Lemma~4.16]{Messing:Crystals}; 
see also \cite[Lemma~2.4.4]{Jong:Crystalline}.
Since we have $\varprojlim(\u A,\sigma)=\u W(R^\flat)$ by Lemma~\ref{Le:lim-A-sigma}, 
the lower limit category is equivalent to $\Win(\u W(R^\flat))$ by \cite[Lemma~2.12]{Lau:Frames}.
\end{proof}

\begin{Thm}
\label{Th:PhiA-equiv}
Let $R$ be a semiperfect ring with a lift $A$ such that the associated
divided powers on $\Ker(\phi)$ given by Remark~\ref{Rk:PD-Ker-phi} are pointwise nilpotent.
Then the functor $\Phi_A$ is an equivalence.
\end{Thm}

\begin{Remark}
If $R$ is a complete intersection or balanced, 
there is a straight lift by Lemma~\ref{Le:exist-lift},
and the associated divided powers on $\Ker(\phi)$ are pointwise nilpotent by Lemma~\ref{Le:PD-nil}.
Thus Theorem~\ref{Th:PhiA-equiv} applies in these cases.
\end{Remark}

\begin{proof}[Proof of Theorem~\ref{Th:PhiA-equiv}]
Since $R^\flat$ is a perfect ring,
the functor $\Phi_{W(R^\flat)}$ is an equivalence by a theorem of Gabber; 
see \cite[Thm.\ 6.4]{Lau:Smoothness}.
Every window over $\u A$ can be lifted to a window over $\u W(R^\flat)$.
Indeed, the projections $A\to R$ and $W(R^\flat)\to R^\flat\to R$ induce bijective maps
of the sets of isomorphisms classes of finite projective modules, and thus the same holds
for $W(R^\flat)\to A$. 
Hence a normal representation of an $\u A$-window in the sense of \cite[Lemma~2.6]{Lau:Frames}
can be lifted to $\u W(R^\flat)$.
Now Lemma~\ref{Le:cat} below 
applied to the diagram of Corollary~\ref{Co:Rflat-R-cart} gives the result.
\end{proof}

\begin{Lemma}
\label{Le:cat}
Let
\[
\xymatrix@M+0.2em{
\AAA \ar[d]_f \ar[r]^\psi &
\CCC \ar[d]^g \\
\BBB \ar[r]^\pi &
\DDD
}
\]
be a cartesian diagram of additive categories or of groupoids.
If $f$ is an equivalence and $\pi$ is essentially surjective,
then $g$ is an equivalence.
\end{Lemma}

\begin{proof}
The case of additive categories is reduced to the case of groupoids using
that a homomorphism $u:X\to Y$ can be encoded by the automorphism
$\left(\begin{smallmatrix}1&0\\u&1\end{smallmatrix}\right)$ of $X\oplus Y$.
Consider the groupoid case. We may assume that $\AAA$ is equal to
the fibered product of $\CCC$ and $\BBB$ over $\DDD$, 
which is the category of triples $(C,\delta,B)$ with
$C\in\CCC$, $B\in\BBB$, and $\delta:g(C)\cong\pi(B)$.

(1) $g$ is surjective on isomorphism classes: This holds for $\pi$ and $f$.

(2) $\psi$ is surjective on isomorphism classes: Let $C\in\CCC$.
Find $B\in\BBB$ and $\delta:g(C)\cong\pi(B)$. 
Then $A=(C,\delta,B)$ satisfies $\psi(A)=C$.

(3) $g$ is faithful: We have to show that if $C\in\CCC$ and $\gamma\in\Aut(C)$
with $g(\gamma)=\id$ then $\gamma=\id$. Extend $C$ to $A=(C,\delta,B)\in\AAA$.
Then $\alpha=(\gamma,\id_B)$ lies in $\Aut(A)$ with $f(\alpha)=\id$.
Thus $\alpha=\id$ and $\gamma=\id$.

(4) $g$ is full: Let $C,C'\in\CCC$ and $\delta:g(C)\cong g(C')$.
Extend $C'$ to $A'=(C',\delta',B')\in\AAA$.
Let $A=(C,\delta'\delta,B')\in\AAA$. 
Then $f(A)=B'=f(A')$, and $\id_{B'}$
lifts to a unique $\alpha:A\cong A'$, which consists of $(\gamma,\id_{B'})$
with $\gamma:C\cong C'$ such that $\delta'\circ g(\gamma)=\delta'\delta$,
thus $g(\gamma)=\delta$.
\end{proof}

\subsection{The passage to $A_{\cris}$}

For a moment let $R$ be an arbitrary semiperfect ring.
By the universal property of $A_{\cris}(R)$
there is a unique lift of $\phi:R\to R$ to a PD endomorphism $\sigma$ of $A_{\cris}(R)$,
and one verifies that $\sigma$ is a Frobenius lift.
Let $\Filone A_{\cris}(R)$ be the kernel of $A_{\cris}(R)\to R$.
By \cite[Lemma~4.1.8]{Scholze-Weinstein} there is a unique functorial $\sigma$-linear map
$\sigma_1:\Filone A_{\cris}(R)\to A_{\cris}(R)$ such that $p\sigma_1=\sigma$, which means that
\begin{equation}
\label{Eq:uAcris}
\u A{}_{\cris}(R)=(A_{\cris}(R),\Filone A_{\cris}(R),R,\sigma,\sigma_1)
\end{equation}
is a $p$-frame, and even a PD frame; see \S\ref{Se:Not}.
A homomorphism of semiperfect rings $R\to R'$ 
induces a strict frame homomorphism $\u A{}_{\cris}(R)\to \u A{}_{\cris}(R')$.

Assume now that $A$ is a lift of $R$ as earlier.
The universal property of $A_{\cris}(R)$ gives a homomorphism
$
\varkappa:A_{\cris}(R)\to A
$
of extensions of $R$, and $\varkappa$ commutes with $\sigma$. 
Since $A$ is torsion free, $\varkappa$ is a frame homomorphism
\begin{equation}
\label{Eq:kappa-Acris-A}
\varkappa:\u A{}_{\cris}(R)\to\u A.
\end{equation}

\begin{Prop}
\label{Pr:str-kappa-cryst}
If $A$ is a straight lift of $R$ in the sense of Definition~\ref{Def:str-lift},
the frame homomorphism $\varkappa$ is crystalline.
\end{Prop}

See also Proposition~\ref{Pr:isobal-kappa-cryst} below.

\begin{proof}
Let $N\subseteq A_{\cris}(R)$ be the kernel of $\varkappa$.
Since $A$ is torsion free we have $N\cap p^nA_{\cris}(R)=p^nN$.
Since $N$ is closed it follows that $N$ is $p$-adically complete.
We have an exact sequence $0\to N/p\to A_{\cris}(R)/p\to R\to 0$,
and this is the PD envelope over $\FF_p$ of the ideal $J=\Ker(R^\flat\to R)$.

Clearly $N$ is stable under $\sigma_1$.
We claim that $\sigma_1$ is nilpotent on $N/p$; cf.\ \cite[Lemma~4.2.4]{Scholze-Weinstein}.
Let $A=W(R^\flat)/J'$. 
The hypothesis means that there are generators $a_i$ of $J$ such that $[a_i]\in J'$.
The ideal $N/pN$ of $A_{\cris}(R)/p$ is generated by the elements $a_i^{[n]}$ for $n\ge 1$.
The elements $[a_i]^{[n]}\in N$ satisfy 
\begin{equation}
\label{Eq:sigma1-Acris}
\sigma_1([a_i]^{[n]})=\frac{(pn)!}{p\cdot n!}[a_i]^{[pn]};
\end{equation}
see \cite[Lemma~4.1.8]{Scholze-Weinstein}.
Since the integer $\frac{(pn)!}{p\cdot n!}$ is divisible by $p$ when $n\ge p$ 
it follows that $\sigma_1\circ\sigma_1=0$ on $N/pN$.

We consider the frames 
$\u B{}_n=(A_{\cris}(R)/p^nN,\Filone A_{\cris}(R)/p^n N, R,\sigma,\sigma_1)$
for $n\ge 0$.
Since $\sigma_1$ is nilpotent on $N/p^nN$, the projection $\u B{}_n\to\u B{}_0=\u A$ 
is crystalline by the general deformation lemma \cite[Th.\ 3.2]{Lau:Frames}.
We have $\varprojlim\u B{}_n=\u A{}_{\cris}(R)$, 
and the proposition follows; see \cite[Le.\ 2.12]{Lau:Frames}.
\end{proof}

\begin{Cor}
\label{Co:BT-Acris-str}
If the semiperfect ring $R$ admits a straight lift $A$, there is an equivalence of categories
\begin{equation}
\label{Eq:BT-Acris-str}
\BT(\Spec R)\cong\Win(\u A{}_{\cris}(R)).
\end{equation}
\end{Cor}

\begin{proof}
By Theorem~\ref{Th:PhiA-equiv} and Proposition~\ref{Pr:str-kappa-cryst}
we have equivalences
\[
\BT(\Spec R)\xrightarrow{\Phi_A\,}\Win(\u A)\xleftarrow{\varkappa^*}\Win(\u A{}_{\cris}(R)).
\qedhere
\]
\end{proof}

\begin{Remark}
\label{Re:BT-Acris-str}
When $A_{\cris}(R)$ is torsion free, the equivalence \eqref{Eq:BT-Acris-str} 
is given by the functor $\Phi_{\u S}$ of \eqref{Eq:PhiS} for $\u S=\u A{}_{\cris}$.
A variant of this holds in general; see Corollary~\ref{Cor:str-PhiR},
which shows in particular that the equivalence \eqref{Eq:BT-Acris-str}
does not depend on the choice of the lift $A$.
\end{Remark}

\begin{Cor}
\label{Co:DR-ci}
If $R$ is a complete intersection semiperfect ring, 
the functor $\DDF_{\Spec R}$ of \eqref{Eq:DFX} is an equivalence.
\end{Cor}

\begin{proof}
If $R$ is a complete intersection, the ring $A_{\cris}(R)$ is torsion free;
see Remark~\ref{Rk:Acris-ci}.
Therefore we have a sequence of functors
\[
\BT(\Spec R)\xrightarrow{\DDF_R}\DF(\Spec R)\xrightarrow e
\Win(A_{\cris}(R))\xrightarrow{\varkappa^*}\Win(\u A)
\]
where $e$ is the evaluation functor, and the composition is $\Phi_A$.
The functor $e$ is an equivalence; see \cite[Prop.\ 2.6.4]{Cais-Lau}.
Here no connection appears because $R^\flat$ is perfect, and thus $\Omega_{R^\flat}=0$.
The functors $\varkappa^*$ and $\Phi_A$ are equivalences by
Theorem~\ref{Th:PhiA-equiv} and Proposition~\ref{Pr:str-kappa-cryst}.
Thus $\DDF_R$ is an equivalence as well.
\end{proof}

\begin{Lemma}
\label{Le:DDF-ci}
If the semiperfect ring $R$ has a lift $A$, 
then the forgetful functor $\D(\Spec R)\to\DF(\Spec R)$ is fully faithful.
\end{Lemma}

\begin{proof}
For a PD extension $S\xrightarrow\pi R$ of $\FF_p$-algebras the Frobenius $\phi_S$
factors through a homomorphism $\phi_{S/R}:R\to S$, i.e.\ $\phi_{S/R}\circ\pi=\phi_S$. 
An object of $\DF(\Spec R)$ is
a triple $(\MMM,F,V)\in\D(\Spec R)$ together with a direct summand of $\MMM_R$ whose base
change under each $\phi_{S/R}$ is determined by $(\MMM,F)$; see \cite[Def.\ 2.4.1]{Cais-Lau}.
The lift $A$ of $R$ makes $\phi:R\to R$ into a PD extension,
which we write as $S\to R$; see Remark~\ref{Rk:PD-Ker-phi}.
The corresponding $\phi_{S/R}$ is the identity of $R$, and the lemma follows.
\end{proof}

Corollary~\ref{Co:DR-ci} together with Lemmas \ref{Le:exist-lift} and \ref{Le:DDF-ci} gives:

\begin{Cor}
(\cite[Cor.\ 4.1.12]{Scholze-Weinstein})
If $R$ is a complete intersection semiperfect ring, the 
crystalline Dieudonn\'e functor $\DD_{\Spec R}$ is fully faithful. \qed
\end{Cor}


\section{Crystalline Dieudonn\'e windows}

In this section we associate to a $p$-divisible group over an arbitrary semiperfect ring $R$
a window over the frame $\u A{}_{\cris}(R)$ of \eqref{Eq:uAcris}.

\subsection{Relative deformation rings}
\label{Se:rel-def-rings}

We need a relative version of the universal deformation of a $p$-divisible group.
Let $\Lambda\to R$ be a homomorphism of $\FF_p$-algebras.
(More generally one could take $p$-adic rings.)

Let $\Aug_{\Lambda/R}$ be the category $\Lambda$-algebras $A$ equipped
with a $\Lambda$-linear homomorphism $A\to R$,
and let $\Nil_{\Lambda/R}\subseteq\Aug_{\Lambda/R}$ be the full subcategory 
of all $A$ such that $A\to R$ is surjective and $J_A=\Ker(A\to R)$ is a nilpotent ideal.
For a $p$-divisible group $G$ over $R$ we consider the deformation functor
\[
\Def_G:\Nil_{\Lambda/R}\to \Set
\]
where $\Def_G(A)$ is the set of isomorphism
classes of deformations of $G$ to $A$.
If $\Lambda=R$, then $\Def_G$ is represented by the twisted power series
ring $B=\Lambda[[Q]]\in\Aug_{\Lambda/R}$ 
with $Q=\Lie(G^\vee)^*\otimes_\Lambda\Lie(G)^*$; 
see \cite[Prop.\ 3.11]{Lau:Relation}.

\begin{Lemma}
\label{Le:deform-bc}
Assume that $G'$ is a $p$-divisible group over $\Lambda$ with an isomorphism 
$G'\otimes_\Lambda R\cong G$.
If $B=\Lambda[[Q]]$ represents $\Def_{G'}:\Nil_{\Lambda/\Lambda}\to\Set$,
then $B$ also represents $\Def_G:\Nil_{\Lambda/R}\to \Set$.
\end{Lemma}

\begin{proof}
For $A\in\Nil_{\Lambda/R}$ the fiber product $A'=A\times_R\Lambda$ lies in
$\Nil_{\Lambda/\Lambda}$.
Let $\LF(A)$ denote the category of finite projective $A$-modules.
Then the obvious functor $\LF(A')\to\LF(A)\times_{\LF(R)}\LF(\Lambda)$
is an equivalence.
It follows that the natural map $\Def_{G'}(A')\to\Def_G(A)$ is bijective,
which proves the lemma.
\end{proof}

Let $\tilde\Nil_{\Lambda/R}$ be the category of all $A\in\Aug_{\Lambda/R}$
such that $A\to R$ is surjective and the ideal $J_A$ is bounded nilpotent, 
i.e.\ there is an $n\ge 1$ with $x^n=0$ for all $x\in J_A$.
We define $\Def_G:\tilde\Nil_{\Lambda/R}\to\Set$ as before.

\begin{Lemma}
\label{Le:deform-bd-nil}
In the situation of Lemma~\ref{Le:deform-bc} the functor $\Def_G$ on $\tilde\Nil_{\Lambda/R}$
is also represented by $B$.
\end{Lemma}

\begin{proof}
Let $A\in\tilde\Nil_{\Lambda/R}$.
We have to show that the natural map $\Hom(B,A)\to\Def_G(A)$ is bijective.
For each pair of homomorphisms $f_1,f_2:B\to A$ in $\Aug_{\Lambda/R}$ there is a
finitely generated ideal $\Fb\subseteq J_A$ such that the projection $A\to\bar A=A/\Fb$ 
equalizes $f$ and $g$. 
For each pair of deformations $G_1,G_2$ of $G$ over $A$ the reduction
map $\Hom_A(G_1,G_2)\to\End_R(G)$ is injective with cokernel annihilated
by $p^r$ for some $r$; see \cite[Lemma~3.4]{Lau:Relation}.
Thus there is a unique isogeny $\psi:G_1\to G_2$ which lifts $p^r\id_G$.
Its kernel is finitely presented; see \cite[Lemma~3.6]{Lau:Relation}.
Thus there is a finitely generated ideal $\Fb\subseteq A$ such that
$\Ker(\psi)$ and $G_1[p^r]$ coincide over $A/\Fb$, 
which means that $G_1$ and $G_2$ map to the same element of $\Def_G(A/\Fb)$.
Moreover $G_1$ and $G_2$ are equal as deformations of $G$ iff they are equal as
deformations of $G_1\otimes_AA/\Fb$.
In view of these remarks it suffices to show that $B(A)\to\Def_G(A)$ is bijecitve 
when $R$ is replaced by $R'=A/\Fb$ for varying finitely generated ideals $\Fb$.
Then $A$ lies in $\Nil_{\Lambda/R'}$, and the lemma follows from Lemma~\ref{Le:deform-bc}.
\end{proof}

\subsection{Construction of the crystalline window functor}

\begin{Thm}
\label{Th:PhiR}
For semiperfect rings $R$ there are unique functors
\[
\Phi^{\cris}_R:\BT(\Spec R)\to\Win(\u A{}_{\cris}(R)),
\qquad
G\mapsto\u M=(M,\Filone M,F,F_1)
\]
which are functorial in $R$, such that the triple $(M,\Filone M,F)$ is given by the
filtered Dieudonn\'e crystal\/ $\DDF(G)$ of \eqref{Eq:DFX} as usual, i.e.\
 $M=\DD(G)_{A_{\cris}(R)}$, the submodule $\Filone M\subseteq M$
is the inverse image of the Hodge filtration $\Lie(G)^*\subseteq\DD(G)_R$,
and $F:M\to M$ is induced by $F:\phi^*\DD(G)\to\DD(G)$.
\end{Thm}

The existence of such a functor has been suggested in \cite[Remark 4.1.9]{Scholze-Weinstein}.
We call $\u M$ the crystalline Dieudonn\'e window of $G$.

\begin{proof}
This is similar to \cite[Th.\ 3.19]{Lau:Relation}.

Let $G\mapsto(M(G),\Filone M(G),F)$ be as defined in the theorem.
We have to find a functorial map $F_1:\Filone M(G)\to M(G)$ which gives a window $\u M(G)$,
and verify that $F_1$ is unique. 
If $A_{\cris}(R)$ is torsion free then $F_1$ and thus $\u M(G)$ are well-defined;
see \cite[Prop.\ 3.17]{Lau:Relation}.
This applies in particular when $R$ is perfect since then $A_{\cris}(R)=W(R)$.

In general let $\pi:R^\flat\to R$ be the projection.
We write $\pi^*$ for the base change functor of modules or windows
from $W(R^\flat)$ to $A{}_{\cris}(R)$.
Note that $p$-divisible groups can be lifted under $\phi:R\to R$ by \cite[thm.\ 4.4]{Illusie:BT}, 
and thus $p$-divisible groups can be lifted under $\pi$.
Let $G\in\BT(\Spec R)$ be given. We choose a lift $G_1\in\BT(\Spec R^\flat)$ of $G$. 
Then $M(G)=\pi^*M(G_1)$ as modules with $\Filone$ and $F$, 
and necessarily we have to define $\u M(G)=\pi^*\u M(G_1)$ as windows. 
We have to show that this construction of $F_1$ does not depend on the choice of $G_1$,
i.e.\ if $G_2\in\BT(\Spec R^\flat)$ is another lift of $G$, then the composite isomorphism
of modules
\[
\pi^*M(G_1)\cong M(G)\cong\pi^*M(G_2)
\]
preserves the $F_1$'s defined on the outer terms by the windows $\u M(G_i)$.

We want to lift the situation to perfect rings. More precisely, we claim that
one can find a commutative diagram of rings
\[
\xymatrix@M+0.2em{
S' \ar[r]^u \ar[d]_f & S \ar[d]^g \\
R^\flat \ar[r]^\pi & R
}
\]
where $S$ and $S'$ are perfect, 
and $p$-divisible groups $H_1,H_2\in\BT(\Spec S')$ 
together with an isomorphism $\alpha:u^*H_1\cong u^*H_2$ over $S$
and isomorphisms $f^*H_i\cong G_i$ over $R^\flat$ for $i=1,2$ such that $\alpha$ induces the given
isomorphism $\pi^*G_1\cong \pi^*G_2$ over $R$, i.e.\ the composition
\[
\pi^*G_1\cong\pi^*f^*H_1\cong g^*u^*H_1\xrightarrow{g^*\alpha}
g^*u^*H_2\cong\pi^*f^*H_2\cong\pi^*G_2
\]
is the given isomorphism.
Then the $F_1$'s of $H_1$ and of $H_2$ coincide over $S$ since $S$ is perfect,
and by base change under $g$ 
it follows that the $F_1$'s of $G_1$ and of $G_2$ coincide over $R$ as required.

Let us prove the claim. 
Let $G'$ be a lift of $G$ to $R^\flat$, for example $G'=G_1$.
Let $B=R^{\flat}[[Q]]$ be the universal deformation ring of $G'$ 
as in \S\ref{Se:rel-def-rings} and let $\GGG$ over $B$ be the universal deformation.
By Lemma~\ref{Le:deform-bd-nil}, $B$ represents the deformation functor 
$\Def_G$ on the category $\tilde\Nil_{R^{\flat}/R}$ of augmented algebras
$R^\flat\to A\to R$ such that the kernel of $A\to R$ is bounded nilpotent.
The system $(\phi^n:R\to R)_n$ is a pro-object of $\tilde\Nil_{R^{\flat}/R}$
with limit $R^\flat\to R$ in $\Aug_{R^\flat/R}$. 
Thus there are homomorphisms $\beta_i:B\to R^\flat$ in $\Aug_{R^\flat/R}$
with $\beta_i^*\GGG\cong G_i$ as deformations of $G$ over $R^\flat$.

We put $S=R^\flat$ with $g=\pi$ and 
$S'=B^{\per}=\varinjlim(B,\phi)$ with $u=\beta_1^{\per}$ and $f=\beta_2^{\per}$.
Let $H_1$ be the base change of $G_1$ under $R^\flat\to B\to S'$ and let $H_2$ be the
base change of $\GGG$ under $B\to S'$. 
Then $u^*H_1\cong G_1\cong u^*H_2$ and $f^*H_1\cong G_1$ and $f^*H_2\cong G_2$ 
as deformations of $G$.
This proves the claim; the required equality of isomorphisms $\pi^*G_1\cong\pi^*G_2$ is
automatic because the reduction map $\Hom(G_1,G_2)\to\End(G)$ is injective.
\end{proof}

The functors $\Phi_R^{\cris}$ 
are related with the functors $\Phi_A$ of \eqref{Eq:PhiA-semiperf} as follows.

\begin{Lemma}
\label{Le:PhiA-PhiR}
If $A$ is a lift of the semiperfect ring $R$, there is a natural isomorphism
$\varkappa^*\circ\Phi_R^{\cris}(G)\cong\Phi_A(G)$,
where $\varkappa$ is defined in \eqref{Eq:kappa-Acris-A}.
\end{Lemma}

\begin{proof}
The functor $\Phi_R^{\cris}$ without $F_1$ is given by the Dieudonn\'e crystal evaluated at $A_{\cris}(R)$.
Thus the functor $\varkappa^*\circ\Phi_R^{\cris}$ without $F_1$ is given by the Dieudonn\'e crystal
evaluated at $A$. 
Since $A$ is torsion free, for the frame $\u A$ the functor of forgetting $F_1$ is fully faithful,
and the lemma follows.
\end{proof}

\begin{Cor}
\label{Cor:str-PhiR}
If the semiperfect ring $R$ admits a straight lift $A$, the functor $\Phi_R^{\cris}$ is an equivalence
and coincides with the equivalence of Corollary~\ref{Co:BT-Acris-str}.
\end{Cor}

\begin{proof}
If $A$ is a straight lift of $R$, the functor $\Phi_A$
is an equivalence by Theorem~\ref{Th:PhiA-equiv} together with Lemma~\ref{Le:PD-nil},
and the functor $\varkappa^*$ is an equivalence by Proposition~\ref{Pr:str-kappa-cryst}.
By Lemma~\ref{Le:PhiA-PhiR} it follows that $\Phi_R^{\cris}$ is an equivalence.
The final assertion is clear.
\end{proof}

\begin{Remark}
\label{Re:str-PhiR}
Corolary \ref{Cor:str-PhiR} is a special case of Theorem~\ref{Th:PhiR-iso-bal} below.
Corollary~\ref{Cor:str-PhiR} applies in particular 
when $R$ is a complete intersection or balanced; see Lemma~\ref{Le:exist-lift}.
For complete intersections, Corollary~\ref{Cor:str-PhiR} is essentially a restatement of
Corollary~\ref{Co:DR-ci}, but the balanced case contains new information.
\end{Remark}

\begin{Cor}
\label{Co:PhiR-iso-bal}
For an iso-balanced semiperfect ring $R$ the crystalline Dieu\-donn\'e functor
$\DD_R:\BT(\Spec R)\to\D(\Spec R)$ is fully faithful up to isogeny.
\end{Cor}

For $f$-semiperfect rings,  this is \cite[Th.~4.1.4]{Scholze-Weinstein}; 
see Remark~\ref{Rk:f-semiperfect}.

\begin{proof}
To prove the assertion
we may replace $R$ by an isogeneous ring; 
see \cite[Prop.\ 4.1.5]{Scholze-Weinstein}. 
Thus we can assume that $R$ is balanced, so $R$ has a straight lift.
For $G\in\BT(\Spec R)$ and $\Phi_R^{\cris}(G)=\u M=(M,\Filone M,F,F_1)$,
the Dieudenn\'e crystal $\DD(G)$ is given by the pair $(M,F)$.
For $G,G'\in\BT(\Spec R)$ we have to show that the composition
\[
\Hom(G,G')\otimes\QQ\to\Hom(\u M,\u M')\otimes\QQ\to\Hom((M,F),(M',F))\otimes\QQ
\]
is bijective. 
The first map is bijective without ${}\otimes\QQ$ by Corollary~\ref{Cor:str-PhiR},
the second map is bijective because the $A_{\cris}(R)$-module $M$ is of finite type.
\end{proof}


\section{The crystalline equivalence}

In this section we extend Corollary~\ref{Cor:str-PhiR} to arbitrary iso-balanced semiperfect rings.
Let $R$ be a semiperfect ring, and let $J=\Ker(R^\flat\to R)$.

\subsection{Weak lifts}

We use a weak version of lifts which may have torsion.

\begin{Defn}
A weak lift of $R$ is a $p$-adically complete ring $A$ with $A/pA=R$ which carries
a ring endomorphism $\sigma:A\to A$ that induces $\phi$ on $R$, and a $\sigma$-linear
map $\sigma_1:pA\to A$ with $\sigma_1(p)=1$.
\end{Defn}

\begin{Remark}
The maps $\sigma$ and $\sigma_1$ are unique if they exist.
This is analogous to Remark~\ref{Rk:lift}:
There is a unique homomorphism $\psi:W(R^\flat)\to A$ of extensions of $R$,
and $\psi$ commutes with $\sigma$. Since $\psi$ is surjective, $\sigma$ is unique.
Then $\sigma_1(px)=\sigma(x)$ is unique as well.
\end{Remark}

By definition, a weak lift $A$ of $R$ gives a PD frame $\u A=(A,pA,R,\sigma,\sigma_1)$.

\begin{Defn}
A weak lift $A$ of $R$ is called straight if $A=W(R^\flat)/J'$ such that $J$
is generated by elements $a$ with $[a]\in J'$.
\end{Defn}

\begin{Lemma}
\label{Le:kappa-str-weak-lift}
For each straight weak lift $A$ of $R$ there is a unique homomorphism of PD frames
\[
\varkappa:\u A{}_{\cris}(R)\to\u A
\]
over the identity of $R$.
\end{Lemma}

\begin{proof}
The universal property of $A_{\cris}(R)$ gives a PD homomorphism $\varkappa:A_{\cris}\to A$
over the identity of $R$, and $\varkappa$ commutes with $\sigma$.
To show that $\varkappa$ is a frame homomorphism it suffices to verify that
$\varkappa(\sigma_1(y))=\sigma_1(\varkappa(y))$ for generators $y$
of the ideal $\Filone A_{\cris}(R)$. 
A set of generators of this ideal is formed by $p$ and the elements $[x]^{[n]}$ 
for generators $x\in J$ and $n\ge 1$.
We have $\varkappa(\sigma_1(p))=1=\sigma_1(\varkappa(p))$, 
moreover \eqref{Eq:sigma1-Acris} gives
\begin{equation}
\label{Eq:kappa-str-weak-lift-1}
\varkappa(\sigma_1([x]^{[n]}))=\frac{(np)!}{p\cdot n!}\varkappa([x]^{[np]})
=\frac{(np)!}{p\cdot n!}\varkappa([x])^{[np]}
\end{equation}
and
\begin{equation}
\label{Eq:kappa-str-weak-lift-2}
\sigma_1(\varkappa([x]^{[n]}))=\sigma_1(\varkappa([x])^{[n]}).
\end{equation}
Since the weak lift $A$ is straight, the generators $x$ of $J$ can be chosen such that $[x]\in J'$.
Then $\varkappa([x])=0$, and
\eqref{Eq:kappa-str-weak-lift-1} and \eqref{Eq:kappa-str-weak-lift-2} are both zero.
\end{proof}

Next we observe that every semiperfect ring has many straight weak lifts.

\begin{Defn}
A descending sequence of ideals $J_0\supseteq J_1\supseteq J_2\supseteq\ldots$ 
of $R^\flat$ is called admissible if $J_0=J$ and $J_i^p\subseteq J_{i+1}$.
In this case let $W(J_*)\subseteq W(R^\flat)$ be the set of all Witt vectors 
$a=(a_0,a_1,a_2,\ldots)$ with $a_i\in J_i$,
which is an ideal by Lemma~\ref{Le:AJ*} below, and let $A(J_*)=W(R^\flat)/W(J_*)$.
\end{Defn}

\begin{Lemma}
\label{Le:AJ*}
Let $J_*$ be an admissible sequence of ideals of $R^\flat$.
Then $W(J_*)$ is an ideal of $W(R^\flat)$, and the ring $A(J_*)$ is a straight weak lift of $R$.
\end{Lemma}

Let $\u A(J_*)=(A(J_*),pA(J_*),R,\sigma,\sigma_1)$ be the corresponding PD frame.

\begin{proof}
As in the proof of Lemma~\ref{Le:exist-lift} we see that $W(J_*)$ is an ideal.
This ideal is closed in $W(R^\flat)$, and thus $A=A(J_*)$ is $p$-adically complete.
Since $J_0=J$ we have $A/pA=R$.
Clearly $W(J_*)$ is stable under the endomorphism $\sigma$ of $W(R^\flat)$, 
so $\sigma$ induces $\sigma:A\to A$.
We have $pA=pW(R^\flat)/(W(J_*)\cap pW(R^\flat))$,
and an element $a\in W(J_*)$ lies in $pW(R^\flat)$ iff $a_0=0$.
Since the sequence $J_*$ is descending we have
$\sigma_1(W(J_*)\cap pW(R^\flat))\subseteq W(J_*)$, 
so $\sigma_1$ induces $\sigma_1:pA\to A$.
It follows that $A$ is a weak lift,
which is straight because for every $a\in J$ we have $[a]\in W(J_*)$.
\end{proof}

Windows over $\u A(J_*)$ are insensitive to bounded variations of $J_*$ :

\begin{Lemma}
\label{Le:AJ*-cryst}
Let $J_*$ and $J'_*$ be two admissible sequences of ideals of $R^\flat$ such
that there is an $n\ge 0$ with $J'_{i+n}\subseteq J_i\subseteq J'_i$ for all $i\ge 0$.
Then there is a natural frame homomorphism $\pi:\u A(J_*)\to\u A(J'_*)$, which is crystalline.
\end{Lemma}

\begin{proof}
The homomorphism $\pi$ exists because $J_i\subseteq J_i'$.
Let $\Fa=\Ker(\pi)=W(J'_*)/W(J_*)$.
Then $\sigma_1$ induces an endomorphism of $\Fa$,
and $(\sigma_1)^n$ is zero on $\Fa$ because $J_{i+n}'\subseteq J_i$.
The result follows from the deformation lemma \cite[Thm.\ 3.2]{Lau:Frames}
if we find a sequence of ideals $\Fa=\Fa_0\supseteq\ldots\supseteq \Fa_n=0$
which are stable under $\sigma_1$ such that $\sigma(\Fa_m)\subseteq \Fa_{m+1}$ for $m<n$.

This sequence can be constructed as follows.
We have $\phi^n(J'_i)\subseteq J_i'^{p^n}\subseteq J'_{i+n}\subseteq J_i$ and thus
$J_i\subseteq J_i'\subseteq\phi^{-n}(J_i)$.
For each $m$ with $0\le m\le n$ let $K_{m,i}=J_i'\cap\phi^{m-n}(J_i)$.
Then $K_{m,*}$ is an admissible  sequence, 
moreover $J'_i=K_{0,i}\supseteq K_{1,i}\supseteq\ldots\supseteq  K_{n,i}=J_i$
and thus $W(J'_*)=W(K_{0,*})\supseteq\ldots\supseteq W(K_{n,*})=W(J_*)$.
Let $\Fa_m=W(K_{m,*})/W(J_*)$.
Then $\Fa_m$ is stable under $\sigma_1$ because $K_{m,*}$ is a decreasing
sequence; see the proof of Lemma~\ref{Le:AJ*}.
We have $\sigma(\Fa_m)\subseteq \Fa_{m+1}$ because $\phi(K_{m,i})\subseteq K_{m+1,i}$.
\end{proof}

\subsection{The passage to $A_{\cris}$}

\begin{Prop}
\label{Pr:isobal-kappa-cryst}
Assume that $R$ is iso-balanced, and let $J_i=J^{p^i}$ for all $i$.
Then the frame homomorphism 
$\varkappa:\u A{}_{\cris}(R)\to\u A(J_*)$ is crystalline.
\end{Prop}

The homomorphism $\varkappa$ is given by Lemma~\ref{Le:kappa-str-weak-lift}.
See also Proposition~\ref{Pr:str-kappa-cryst}. 

\begin{proof}
Let $R\to R'$ be an isogeny with balanced $R'$ whose kernel is annihilated by $\phi^n$.
Then $R'=R^\flat/J'$ with $\phi(J')=J'^p$ and $\phi^n(J')\subseteq J\subseteq J'$.
Let $K_i=J\cap\phi^i(J')$ for $i\ge 0$. The sequence $K_*$ is admissible. 
For $i\ge 0$ we have $K_{n+i}=\phi^{n+i}(J')\subseteq J_i\subseteq K_i$.
Thus the natural frame homomorphism $\pi:\u A(J_*)\to\u A(K_*)$
is crystalline by Lemma~\ref{Le:AJ*-cryst},
and it suffices to show that the composition
\[
\varkappa'=\pi\circ\varkappa:\u A{}_{\cris}(R)\to\u A(K_*)
\]
is crystalline. Let $N=\Ker(\varkappa')$.

\begin{Lemma}
\label{Le:N} 
\begin{enumerate}
\item
\label{Le:N-1}
The torsion of $A(K_*)$ is annihilated by $p^n$.
\item
\label{Le:N-2}
For $i\ge 0$ we have $N\cap p^{n+i}A_{\cris}(R)\subseteq p^iN$,
in particular the $p$-adic topology of\/ $N$ is induced by
the $p$-adic topology of $A_{\cris}(R)$.
\item
\label{Le:N-3}
The endomorphism $\sigma_1:N/pN\to N/pN$ is nilpotent.
\end{enumerate}
\end{Lemma}

\begin{proof}
Let $J'_i=\phi^{i}(J')$. 
Then $J'_*$ is an admissible sequence of ideals of $W(R^\flat)$
with respect to $R'=R^\flat/J'$; note that $R^\flat=R'^\flat$.
We have $K_i\subseteq J'_i$ with equality for $i\ge n$,
so there is a projection $A(K_*)\to A(J'_*)$ whose kernel is annihilated by $p^n$. 
The ring $A(J'_*)$ is the straight lift of $R'$ constructed in Lemma~\ref{Le:exist-lift}, 
which is torsion free. This proves \eqref{Le:N-1}, and \eqref{Le:N-2} follows.

Let us prove \eqref{Le:N-3}.
The ring  $A_{\cris}(R)$ is the $p$-adic completion of a $W(R^\flat)$-algebra
generated by the elements $[x]^{[i]}$ for $x\in J$ and $i\ge 1$, and these elements
map to zero in $A(K_*)$.
Thus for each $m\ge 1$ the image of $N$ in $A_{\cris}(R)/p^m$
is generated as an ideal by $W(K_*)$ and the elements $[x]^{[i]}$.
By \eqref{Le:N-2} it follows that $N/pN$ is generated as an $A_{\cris}(R)$-module
by $W(K_*)$ and the elements $[x]^{[i]}$. We check these elements separately.

First, the explicit formula \eqref{Eq:sigma1-Acris} for $\sigma_1$ implies that
for $x\in J$ the element $(\sigma_1)^2([x]^{[i]})$ lies in $pN$; 
see the proof of Proposition~\ref{Pr:str-kappa-cryst}.
Second, since $\phi^i(K_{n})=K_{n+i}$ for $i\ge 0$, each element of $W(K_*)/pW(K_*)$
is represented by an element $a=(a_0,a_1,a_2,\ldots)\in W(K_*)$ with $a_{i}=0$ for $i>n$.
Then
\[
(\sigma_1)^{n+2}(a)=
(\sigma_1)^{n+2}([a_0])+\ldots+(\sigma_1)^2([a_{n}])
\]
lies in $pN$, using that $a_i\in K_i\subseteq J$. 
Thus $(\sigma_1)^{n+2}$ is zero on $N/pN$, 
and Lemma~\ref{Le:N} is proved.
\end{proof}

We continue the proof of Proposition~\ref{Pr:isobal-kappa-cryst}.
Since $A(K_*)$ is $p$-adically complete, the ideal $N$ is closed in $A_{\cris}(R)$,
and thus $N$ is $p$-adically complete by Lemma~\ref{Le:N} \eqref{Le:N-2}.
Since $\sigma_1:N\to N$ stabilizes $p^mN$, the ring $A_{\cris}(R)/p^mN$
carries a natural frame structure, denoted by $\u A{}_{\cris}(R)/p^mN$.
We have $\u A{}_{\cris}(R)/N=\u A(K_*)$ and 
$\u A{}_{\cris}(R)=\varprojlim\u A{}_{\cris}(R)/p^mN$.
This limit preserves the window categories by \cite[Lemma~2.12]{Lau:Frames}.
Thus it suffices to show that the frame homomorphism $\u A{}_{\cris}(R)/p^mN\to\u A(K_*)$ 
is crystalline for each $m$.
Since $\sigma_1:N/p^mN\to N/p^mN$ is nilpotent by Lemma~\ref{Le:N} \eqref{Le:N-3},
this follows from \cite[Thm.\ 3.2]{Lau:Frames}.
\end{proof}

\begin{Thm}
\label{Th:PhiR-iso-bal}
If $R$ is an iso-balanced semiperfect ring, the functor 
$\Phi_R^{\cris}$ of Theorem~\ref{Th:PhiR} is an equivalence of categories.
\end{Thm}

\begin{proof}
By Corollary~\ref{Cor:str-PhiR} the theorem holds for balanced rings. 
An isogeny from $R$ to a balanced ring has nilpotent kernel by Lemma~\ref{Le:ker-hom-bal}.
Therefore it suffices to show: Let $\pi:R'\to R$ be an isogeny of iso-balanced rings
such that $\Ker(\pi)^p=0$. If $\Phi_R^{\cris}$ is an equivalence then so is $\Phi_{R'}^{\cris}$.

To prove this we use some auxiliary frames.
Let $J=\Ker(R^\flat\to R)$ and $J'=\Ker(R^\flat\to R')$, thus $J^p\subseteq J'\subseteq J$.
We define $J_i=J^{p^i}$ and $J'_i=J'^{p^i}$ for $i\ge 0$, 
and we define $K_0=J'$ and $K_i=J_i$ for $i\ge 1$.
Then $J_*$ is an admissible sequence with respect to $R$, while $K_*$ and $J'_*$ are
admissible sequences with respect to $R'$.
There are obivous frame homomorphisms
\[
\u A(J'_*)\xrightarrow a \u A(K_*) \xrightarrow q \u A(J_*),
\]
where $a$ lies over $\id_{R'}$ and $q$ lies over $\pi$.
Here $a$ is crystalline by Lemma~\ref{Le:AJ*-cryst},
using that $K_{i+1}\subseteq J'_i\subseteq K_i$.
We want to factor $q$ over another frame $\FFF=(A,I,R,\sigma,\sigma_1)$ with $A=A(K_*)$,
thus $I$ is the kernel of $A(K_*)\to A(J_*)\to R$. 
We only have to define $\sigma_1:I\to A$.
It is easy to see that the natural map $\Ker(q)\to\Ker(\pi)$ is bijective and that
\[
I=\Ker(q)\oplus pA
\]
as a direct sum of ideals. We have $\Ker(q)^p=0$, and $\sigma(x)=0$ for $x\in \Ker(q)$.
We extend the homomorphism $\sigma_1$ and the divided powers defined on $pA$
to $I$ by $\sigma_1(x)=0$ and $x^{[p]}=0$ for $x\in \Ker(q)$.
This defines a PD frame $\FFF$ as above.
Together we have homomorphisms of PD frames
\[
\u A(J'_*)\xrightarrow a \u A(K_*)\xrightarrow b\FFF\xrightarrow c\u A(J_*)
\]
over $R'\xrightarrow\id R'\xrightarrow\pi R\xrightarrow\id R$,
where $c$ is given by $q$ and $b$ is given by $\id_A$.
Since $\sigma_1$ is zero on $\Ker(c)=\Ker(q)$,
 $c$ is crystalline by \cite[Thm.\ 3.2]{Lau:Frames}.

Since $A\to R$ is a $p$-adic PD extension, the universal property of $A_{\cris}(R)$
gives a unique homomorphism $\tilde\varkappa:A_{\cris}(R)\to A$ of PD extensions of $R$,
and $\tilde\varkappa$ commutes with $\sigma$.
We claim that $\tilde\varkappa$ is a frame homomorphism $\u A{}_{\cris}(R)\to\FFF$,
i.e.\ that $\tilde\varkappa$ commutes with $\sigma_1$.
As in the proof of Lemma~\ref{Le:kappa-str-weak-lift} it suffices to show that
$\tilde\varkappa(\sigma_1(y))=\sigma_1(\tilde\varkappa(y))$ when $y=[x]^{[n]}$
with $x\in J$ and $n\ge 1$.
Let $z=\tilde\varkappa([x])$ in $A$. 
Then $z\in\Ker(q)$ and thus $z^{[np]}=0$,
moreover $z^{[n]}\in\Ker(q)$ as well and thus $\sigma_1(z^{[n]})=0$.
Therefore \eqref{Eq:kappa-str-weak-lift-1} and \eqref{Eq:kappa-str-weak-lift-2}
with $\tilde\varkappa$ in place of $\varkappa$ show that 
$\tilde\varkappa(\sigma_1(y))=0$ and $\sigma_1(\tilde\varkappa(y))=0$.
Thus $\tilde\varkappa$ is a frame homomorphism.
Since $a,b,c$ are PD homomorphisms, we obtain a commutative diagram of frames
\[
\xymatrix@M+0.2em{
\u A{}_{\cris}(R') \ar[rrr]^{A_{\cris}(\pi)} \ar[d]_\varkappa \ar[dr]^{\varkappa'} &&&
\u A{}_{\cris}(R) \ar[d]^\varkappa \ar[dl]_{\tilde\varkappa} \\
\u A(J'_*) \ar[r]^-a &
\u A(K_*) \ar[r]^-b &
\FFF \ar[r]^-c &
\u A(J_*)
}
\]
where the homomorphisms $\varkappa$ are given by Lemma~\ref{Le:kappa-str-weak-lift},
and $\varkappa'=a\circ\varkappa$.
The two homomorphisms $\varkappa$ are crystalline by Proposition~\ref{Pr:isobal-kappa-cryst}.
Since $a$ and $c$ are crystalline, the same holds for $\varkappa'$ and $\tilde\varkappa$.
We have the following commutative diagram of categories.
\[
\xymatrix@M+0.2em{
\BT(\Spec R') \ar[r]^-{\Phi_{R'}^{\cris}} \ar[d]_{\pi^*} &
\Win(\u A{}_{\cris}(R')) \ar[r]^-{\varkappa'}_-\sim \ar[d] &
\Win(\u A(K_*)) \ar[d]^{b} \\
\BT(\Spec R) \ar[r]^-{\Phi_{R}^{\cris}} &
\Win(\u A{}_{\cris}(R)) \ar[r]^-{\tilde\varkappa}_-\sim &
\Win(\FFF)
}
\]
The functors $\BT(\Spec R')\to\Win(\u A(K_*))$ and $\BT(\Spec R)\to\Win(\FFF)$ are given by
the Dieudonn\'e crystal with an additional $F_1$.
Since $\Ker(\pi)^p=0$, the ideal $\Ker(\pi)$ can be equipped with the trivial divided powers.
Then the projection $A(K_*)=A\to R'$ is a homomorphism of PD extensions of $R$.
It follows that for $G\in\BT(\Spec R)$ with associated $\u M\in\Win(\FFF)$
there is a natural isomorphism $M\otimes_AR'\cong\DD(G)_{R'}$.
By the Grothendieck-Messing Theorem \cite{Messing:Crystals} and its trivial
counterpart for the frame homomorphism $b$ in \cite[Lemma~4.2]{Lau:Frames}
it follows that the lifts of $G$ under $\pi$ and the lifts of $\u M$ under $b$ coincide;
cf.\ the proof of Proposition~\ref{Pr:Cartesian-sigma}.
Therefore if $\Phi_R^{\cris}$ is an equivalence, the same holds for $\Phi_{R'}^{\cris}$.
This finishes the proof of Theorem~\ref{Th:PhiR-iso-bal}.
\end{proof}


\section{Perfectoid rings}

We use the definition of perfectoid rings of \cite{BMS}
in a slightly different formulation.
We begin with an easy remark on perfect rings.

\begin{Lemma}
\label{Le:IJ-perf}
Let $S$ be a perfect ring and $a\in S$.
Let $J=(a^{p^{-\infty}})$ and $I=\Ann(a)$.
Then $J=\phi(J)$ und $I=\phi(I)=\Ann(J)$ and $I\cap J=0$,
 thus we have an exact sequence
\begin{equation}
\label{Eq:IJ-perf}
0\to S \to S/I\oplus S/J\to S/(I+J)\to 0
\end{equation}
where the first map is the diagonal map and the second map is the difference.
The element $a\in S/I$ is a non-zero divisor.
If $S$ is $a$-adically complete, the same holds for $S/I$.
\end{Lemma}

\begin{proof}
Clearly $J=\phi(J)$, moreover $I=\Ann(a)\subseteq\Ann(a^p)=\phi(I)\subseteq I$,
thus $I=\phi(I)=\Ann(J)$. Since $S$ is reduced we have $I\cap J=IJ=0$,
and \eqref{Eq:IJ-perf} is exact.
Since $I=\Ann(a)=\Ann(a^2)$ the element $a\in S/I$ is a non-zero divisor.
The last assertion follows from \eqref{Eq:IJ-perf}
because $S/J$ and $S/(I+J)$ are annihilated by $a$.
\end{proof}

The exact sequence \eqref{Eq:IJ-perf} can also be expressed by the
cartesian and cocartesian diagram of perfect rings:
\begin{equation}
\label{Eq:IJ-square}
\xymatrix@M+0.2em{
S \ar[r] \ar[d] & S/J \ar[d] \\
S/I \ar[r] & S/(I+J)
}
\end{equation}

The following is contained in \cite[Prop.\ 7.3.45]{Gabber-Ramero:Foundations}. 

\begin{Lemma}
\label{Le:xi-regular}
Let $S$ be a perfect ring and let $\xi=(\xi_0,\xi_1,\ldots)\in W(S)$ 
such that $\xi_0,\ldots,\xi_r$ generate the unit ideal of $S$. 
Then $\xi$ is a non-zero divisor, and for $n\ge 0$ we have
\begin{equation}
\label{Eq:xi-regular}
\xi W(S)\cap p^{n+r}W(S)=p^n(\xi W(S)\cap p^rW(S)).
\end{equation} 
In particular, $\xi W(S)$ is $p$-adically closed in $W(S)$ and $p$-adically complete.
\end{Lemma}

\begin{proof}
Using Lemma~\ref{Le:IJ-perf} with $a=\xi_0$
one reduces to the case where $\xi_0=0$ or where $\xi_0$ is a non-zero divisor.
In the second case $\xi$ is a non-zero divisor, and $\xi W(S)\cap p^nW(S)=p^n\xi W(S)$.
If $\xi_0=0$ then $\xi=p\xi'$, and the proof of \eqref{Eq:xi-regular}
is finished by induction on $r$. The last assertion follows easily.
\end{proof}

\begin{Defn}
\label{Def:distd}
For a perfect ring $S$, 
an element $\xi=(\xi_0,\xi_1,\ldots)\in W(S)$ is called distinguished
if $\xi_1\in S$ is a unit and $S$ is $\xi_0$-adically complete.
\end{Defn}

\begin{Remark}
If a ring $R$ is complete with respect to some linear topology and $x\in R$ is
topologically nilpotent, then $R$ is also $x$-adically complete; 
see the proof of \cite[\href{http://stacks.math.columbia.edu/tag/090T}{Tag 090T}]{Stacks-Project}. 
\end{Remark}

\begin{Defn}
\label{Def:perfd}
A ring $\pfd$ is called perfectoid if there is an isomorphism 
$\pfd\cong W(S)/\xi$ where $S$ is perfect and $\xi\in W(S)$ is distinguished.
\end{Defn}

\begin{Remark}
\label{Rk:Def-perfd}
Definition~\ref{Def:perfd} is equivalent to \cite[Definition~3.5]{BMS},
moreover for $\pfd=W(S)/\xi$ as in Definition~\ref{Def:perfd} we have 
\begin{equation}
\label{Eq:S-Bflat}
S=\pfd^\flat:=\varprojlim(\pfd/p,\phi)
\end{equation}
canonically. Indeed, if $\pfd=W(S)/\xi$ then $\pfd/p=S/\xi_0$,
the projective system $\pfd/p\leftarrow \pfd/p\leftarrow\ldots$ with arrows $\phi$
is identified with $S/\xi_0\leftarrow S/\xi_0^p\leftarrow \ldots$ 
where the arrows are the projection maps, and \eqref{Eq:S-Bflat} follows 
since $S$ is $\xi_0$-adically complete.
Moreover, $\pfd$ is $p$-adically complete because this holds for $W(S)$
and because $\xi W(S)$ is $p$-adically closed by Lemma~\ref{Le:xi-regular}.
If $\pi\in \pfd$ is the image of $[\xi_0^{1/p}]\in W(S)$ then $\pi^p\pfd=p\pfd$.
Thus $\pfd$ satisfies \cite[Definition~3.5]{BMS}.
Conversely, if the latter holds,
then $\pfd=W(\pfd^\flat)/\xi$ where $\pfd^\flat$ is perfect and $\xi$ is distinguished.
See also \cite[14.1.32]{Gabber-Ramero:Foundations}.
\end{Remark}

\begin{Remark}
\label{Rk:perfd-torsion-free}
The perfectoid ring $\pfd=W(S)/\xi$ is torsion free iff $\xi_0\in S$ is a non-zero divisor.
Indeed, since $p,\xi\in W(S)$ are regular elements, the kernels of
$p:\pfd\to \pfd$ and of $\xi_0:S\to S$ are isomorphic.
\end{Remark}

\begin{Remark}
\label{Rk:IJ-perfd}
If $\pfd=W(S)/\xi$ is perfectoid, the decomposition \eqref{Eq:IJ-square}
of $S$ with respect to $a=\xi_0$ gives a similar decomposition of $\pfd$.
More precisely, let $S_1=S/{\Ann(\xi_0)}$ and $S_2=S/(\xi_0^{p^{-\infty}})$ 
and $S_{12}=S_1\otimes_SS_2$.
Then $\xi\in W(S_i)$ is distinguished, and $\pfd_{i}=W(S_{i})/\xi$ is perfectoid.
\eqref{Eq:IJ-perf} gives an exact sequence 
\begin{equation}
\label{Eq:IJ-perfd-WS}
0\to W(S)\to W(S_1)\oplus W(S_2)\to W(S_{12})\to 0.
\end{equation}
Since $\xi$ is a non-zero divisor in $W(S_{12})$, we obtain an exact sequence
\begin{equation}
\label{Eq:IJ-perfd-B}
0\to \pfd \to \pfd_1\oplus \pfd_2\to \pfd_{12}\to 0.
\end{equation}
Here $\pfd_2=S_2$ and $\pfd_{12}=S_{12}$ are perfect, while $\pfd_1$ is torsion free perfectoid.
\end{Remark}

As an easy consequence we observe:

\begin{Lemma}
Every perfectoid ring $\pfd$ is reduced.
\end{Lemma}

\begin{proof}
By \eqref{Eq:IJ-perfd-B} we can assume that $\pfd$ is either perfect (thus reduced) or torsion free.
For $\pi\in \pfd$ as in Remark~\ref{Rk:Def-perfd} we have $\pi^p\pfd=p\pfd$, 
and $\phi:\pfd/\pi\to \pfd/p$ is bijective.
Hence, if $a\in \pfd$ satisfies $a^p=0$ then $a=\pi b$.
If $\pfd$ is torsion free it follows that $b^p=0$,
thus $a\in\pi^n\pfd$ for all $n$, whence $a=0$.
\end{proof}

We need the following form of tilting.

\begin{Lemma}
\label{Le:tilt-etale}
Let $\pfd$ be a perfectoid ring and $B=\pfd/p$. 
The functor $\pfd'\mapsto \pfd'/p$ from perfectoid $\pfd$-algebras to $B$-algebras 
has a left adjoint $B'\mapsto B'^\sharp$.
If $B'$ is an etale $B$-algebra then $\pfd'=B'^\sharp$ is the unique $p$-adically complete
$\pfd$-algebra such that $\pfd'/p=B'$ and $\pfd/p^n\to \pfd'/p^n$ is etale for all $n$.
\end{Lemma}

\begin{proof}
Let $\pfd=W(S)/\xi$ where $\xi$ is distinguished, 
thus $S=B^\flat=\varprojlim(B,\phi)$,
see Remark~\ref{Rk:Def-perfd}.
For a $B$-algebra $B'$ let $B'^\sharp=W(B'^\flat)/\xi$.
This defines the left adjoint functor. 
Assume that $B\to B'$ is etale and let $\pfd'=B'^\sharp$.
We have to show that $\pfd'/p=B'$ and that $\pfd/p^n\to \pfd'/p^n$ is flat.
Let $x_n\in\pfd/p$ be the image of $\xi_0^{1/p^n}$,
so $x_n(\pfd/p)$ is the kernel of $\phi^n:\pfd/p\to\pfd/p$.
Since $B\to B'$ is etale, the diagram of rings
\[
\xymatrix@M+0.2em{
B \ar[r] \ar[d]_{\phi^n} & B' \ar[d]^{\phi^n} \\ B \ar[r]  & B' 
}
\]
is cocartesian, in particular $x_nB$ is the kernel of $\phi^n:B'\to B'$.
It follows that $B'=B'^\flat/\xi_0$ (see EGA $0_{\text I}$ Prop.\ 7.2.7) 
and thus $\pfd'/p=B'^\flat/\xi_0=B'$.
We have $\pfd/p^n=W(B^\flat)/([\xi_0^n],p^n,\xi)$ and similarly for $\pfd'$.
For fixed $n$ let 
\[
C=W(B^\flat)/([\xi_0^n],p^n),
\qquad
C'=W(B'^\flat)/([\xi_0^n],p^n).
\]
In order to verify that $\pfd/p^n\to \pfd'/p^n$ is flat it suffices to show that $C\to C'$ is flat,
or equivalently that $C/p\to C'/p$ is flat and that
the associated graded rings satisfy $\gr_p(C')=\gr_p(C)\otimes_{C/p}C'/p$
(local flatness criterion).
But $C/p=B^\flat/\xi_0^n\cong B/x_r^n$ when $p^r\ge n$, and $gr_p(C)\cong (C/p)[T]/T^n$;
and similarly for $C'$. The assertion follows.
\end{proof}

\subsection{The ring $A_{\cris}$ for perfectoid rings}

Let $\pfd=W(S)/\xi$ be a perfectoid ring where $S$ is perfect and $\xi$ is distinguished.
Let $A_{\inf}(\pfd)=W(S)$ 
and let $A_{\cris}(\pfd)\to \pfd$ be the universal $p$-adic PD extension.
We have $A_{\cris}(\pfd)=A_{\cris}(\pfd/p)$ as rings.
If $\pfd$ is perfect then $A_{\cris}(\pfd)=A_{\inf}(\pfd)=W(\pfd)$.
If $\pfd$ is torsion free, which means that the semiperfect ring $\pfd/p$ is
a complete intersection in the sense of Definition \ref{Def:ci} 
(see Remark~\ref{Rk:perfd-torsion-free}), 
then $A_{\cris}(\pfd)$ is torsion free. This also holds in general:

\begin{Prop}
\label{Pr:IJ-Acris}
Let $\pfd=W(S)/\xi$ be a perfectoid ring as above 
and $\pfd_i=W(S_i)/\xi$ as in Remark~\ref{Rk:IJ-perfd}, for $i=1,2,12$.
We have an exact sequence
\[
0\to A_{\cris}(\pfd)\to A_{\cris}(\pfd_{1})\oplus W(\pfd_{2})\to W(\pfd_{12})\to 0. 
\]
In particular, the ring $A_{\cris}(\pfd)$ is torsion free.
\end{Prop}

\begin{proof}
Recall that 
\[
S_1=S/{\Ann(\xi_0)},\quad
S_2=S/(\xi_0^{p^{-\infty}})=\pfd_2,\quad
S_{12}=S_1\otimes_SS_2=\pfd_{12}.
\]
For $i=\emptyset$ or $1$ let $A_i$ be the PD envelope of $[\xi_0]W(S_i)\subseteq W(S_i)$ 
relative to $p\ZZ_p\subset\ZZ_p$. Then $A_{\cris}(\pfd_i)$ is the $p$-adic completion of $A_i$. 
The projection $W(S_i)\to W(S_{i2})$ extends to a PD homomorphism $g_i:A_i\to W(S_{i2})$,
and we have a commutative diagram of rings, where $f_i$ is the canonical map:
\[
\xymatrix@M+0.2em{
W(S) \ar[r]^-f \ar[d] & A \ar[r]^-g \ar[d] & W(S_2) \ar[d] \\
W(S_1) \ar[r]^-{f_1} & A_1 \ar[r]^-{g_1} & W(S_{12})
}
\]
We claim that  $\Coker(f)\to\Coker(f_1)$ is bijective, $f_1$ is injective, and $A_1$ is torsion free.
Assume this holds. The diagonal map $W(S)\to W(S_1)\times W(S_2)$ is injective, 
thus $W(S)\to W(S_1)\times A$ is injective. 
Since $f_1$ is injective, it follows that $f$ is injective.
Consider the homomorphisms of complexes
\[
[W(S)\to W(S_1)]\xrightarrow{f_*}[A\to A_1]\xrightarrow{g_*}[W(S_2)\to W(S_{12})].
\]
Here $f_*$ and $g_*\circ f_*$ are quasi-isomorphism, thus $g_*$ is a quasi-isomorphism.
This remains true after $p$-adic completion, and the lemma follows.

To prove the claim we need a closer look on the construction of $A$ and $A_1$.
Let $\Lambda_0=\ZZ_p[T]$ 
and let $\Lambda=\ZZ_p{\left<T\right>}$ be the PD polynomial algebra, 
i.e.\ the $\ZZ_p$-subalgebra of $\QQ_p[T]$ generated by $T^n/n!$ for $n\ge 1$.
Define $\Lambda_0\to W(S)$ by $T\mapsto[\xi_0]$.
This extends to a PD homomorphism $\Lambda\to A$, and the resulting homomorphisms
\[
h:W(S)\otimes_{\Lambda_0}\Lambda\to A,
\qquad
h_1:W(S_1)\otimes_{\Lambda_0}\Lambda\to A_1
\]
are surjective.
Since $\xi_0$ is a non-zero divisor in $S_1$, 
the homomorphism $h_1$ is bijecitve, and $A_1$ is torison free.\footnote{
In more detail:
Since $[\xi_0]$ is not a zero divisor in $W(S_1)$, 
the ring $A_1'=W(S_1)\otimes_{\Lambda_0}\Lambda$ is the absolute PD envelope 
of $[\xi_0]W(S_1)\subseteq W(S_1)$; see \cite[p.\ 64, (3.4.8)]{Berthelot:CohCristalline}.
Since $\xi_0$ is not a zero divisor in $S_1$ we have $\Tor_1^{\Lambda_0}(W(S_1),\FF_p)=0$.
Then $\Tor_1^{\Lambda_0}(W(S_1),\Lambda/p)=0$ because the $\Lambda_0$-module
$\Lambda/p$ is isomorphic to the direct sum of infinitely many copies of $\FF_p[T]/T^p$.
Since $\Lambda$ is $p$-torsion free it follows that $A_1'$ is $p$-torsion free, 
and therefore $A_1'=A_1$.}

We consider the following ascending filtration of $\Lambda$ 
and the associated filtrations of $A$ and of $A_1$.
For $m\ge 0$ let $F^m\Lambda=\Lambda\cap p^{-m}\ZZ_p[T]$.
Then $\Lambda=\bigcup F^m\Lambda$, 
and ${\gr^m}\Lambda=F^m\Lambda/F^{m-1}\Lambda$ 
is a free $\Lambda_0/p$-module of rank $1$ generated by $p^{-m}T^{d_m}$ 
where $d_m$ is minimal such that $p^m$ divides $d_m!$. 
For $i=\emptyset$ or $1$ let $F^mA_i\subseteq A_i$ be the image of 
$W(S_i)\otimes_{\Lambda_0} F^m\Lambda$ and let ${\gr^m}A_i=F^mA_i/F^{m-1}A_i$.
The homomorphism $h_i$ induces surjective maps
\[
F^mh_i:W(S_i)\otimes_{\Lambda_0}F^m\Lambda\to F^mA_i
\]
and surjective maps
\[
{\gr^m}h_i:S_i\cong W(S_i)\otimes_{\Lambda_0}{\gr^m}\Lambda\to{\gr^m}A_i
\]
which map $1\in S_i$ to $(p^{-m}d_m!)\gamma_{d_m}([\xi_0])$.
The transition homomorphisms 
\[
W(S_i)\otimes_{\Lambda_0}F^{m-1}\Lambda\to W(S_i)\otimes_{\Lambda_0}F^m\Lambda
\] 
are injective because 
\[
\Tor_1^{\Lambda_0}(W(S_i),{\gr^m}\Lambda)\cong\Tor_1^{\ZZ_p}(W(S_i),\FF_p)=0.
\]
Since $h_1$ is bijective it follows that $F^mh_1$ is bijective for $m\ge 0$, 
which for $m=0$ means that $f_1:W(S_1)\to A_1$ is injective,
moreover ${\gr^m}h_1$ is bijective for $m\ge 1$. 
We consider the commutative diagram of surjective maps
\[
\xymatrix@M+0.2em{
S \ar[r] \ar[d]_{{\gr^m}h} &
S_1 \ar[d]^{{\gr^m}h_1}_\cong \\
{\gr^m} A \ar[r] &
{\gr^m} A_1
}
\]
We claim that $\Ker(S\to S_1)=\Ann(\xi_0)$ maps to zero in ${\gr^m}A$.
Indeed, choose $r$ such that $p^r\ge d_m$.
For $a\in\Ann(\xi_0)$ we have $b=a^{p^{-r}}\in\Ann(\xi_0)$ and therefore 
$[a]\gamma_{d_m}([\xi_0])=[b^{p^r-d_m}]\gamma_{d_m}([b\xi_0])=0$.
It follows that ${\gr^m}A\to{\gr^m}A_1$ is bijective for $m\ge 1$,
and thus $A/F^0A\to A_1/F^0A_1$ is bijective,
which means that $\Coker(f)\to\Coker(f_1)$ is bijective as required.
\end{proof}

\section{Windows and modules for perfectoid rings}

As earlier let $\pfd=W(S)/\xi$ be a perfectoid ring where $S$ is perfect and $\xi$ is distinguished.
The rings $A_{\inf}(\pfd)=W(S)$ and $A_{\cris}(\pfd)$ carry natural frame structures:
\[
\u A{}_{\inf}(\pfd)=(A_{\inf}(\pfd),\Filone A_{\inf}(\pfd),\pfd,\sigma,\sigma_1^{\inf})
\]
where $\Filone A_{\inf}(\pfd)=\xi A_{\inf}(\pfd)$ and $\sigma_1^{\inf}(\xi a)=\sigma(a)$, and
\[
\u A{}_{\cris}(\pfd)=(A_{\cris}(\pfd),\Filone A_{\cris}(\pfd),\pfd,\sigma,\sigma_1)
\]
where $\Filone A_{\cris}(\pfd)$ is the kernel of $A_{\cris}(\pfd)\to \pfd$, 
and $\sigma_1(a)=p^{-1}\sigma(a)$; this is well-defined since $A_{\cris}(\pfd)$
is torsion free by Proposition~\ref{Pr:IJ-Acris}.
The natural map $A_{\inf}(\pfd)\to A_{\cris}(\pfd)$ is a frame homomorphism
\begin{equation}
\label{Eq:lambda}
\lambda:\u A{}_{\inf}(\pfd)\to\u A{}_{\cris}(\pfd).
\end{equation}
Indeed, let $c=\sigma_1(\xi)$ in $A_{\cris}(\pfd)$.
Then $c\equiv[\xi_0]^p/p+[\xi_1]^p$ mod $pA_{\cris}(\pfd)$ and thus $c\equiv [\xi_1]^p$
mod $pA_{\cris}(\pfd)+\Filone A_{\cris}(\pfd)$, so $c$ is a unit since $\xi_1$ is a unit.
We have $\sigma_1\circ\lambda=c\cdot\lambda\circ\sigma_1$ on $\xi A_{\inf}(R)$,
so $\lambda$ is a $c$-homomorphism of frames in the sense of \cite{Lau:Frames}.
If $\pfd$ is perfect, $\lambda$ is the identity and $c=1$.

\subsection{Descent of windows under $\lambda$}

We need the following standard lemma.
For a ring $A$ let\ $\LF(A)$ be the category of finite projective $A$-modules.

\begin{Lemma}
\label{Le:LF-A12}
Let $A_1\to A_3\leftarrow A_2$ be rings with surjective homomorphisms and 
$A=A_1\times_{A_3}A_2$. Then the corresponding diagram of categories
\[
\xymatrix@M+0.2em{
\LF(A) \ar[r] \ar[d] & \LF(A_2) \ar[d] \\ \LF(A_1) \ar[r] & \LF(A_3)
}
\]
is $2$-cartesian.
\end{Lemma}

\begin{proof}
For a flat $A$-module $M$ and $M_i=M\otimes_AA_i$ 
the natural map $M\to M_1\times_{M_3}M_2$ is bijective. 
Thus the functor $\LF(A)\to\LF(A_1)\times_{\LF(A_3)}\LF(A_2)$ is fully faithful.
For given $M_i\in\LF(A_i)$ and isomorphisms 
$M_1\otimes_{A_1}A_3\cong M_3\cong M_2\otimes_{A_2}A_3$ let $M=M_1\times_{M_3}M_2$.
We have to show that $M\in\LF(A)$ and that $M\otimes_AA_i\to M_i$ is bijective.
One can choose a finite free $A$-module $F$ and compatible surjective maps $g_i:F_i\to M_i$ 
where $F_i=F\otimes_AA_i$. 
Indeed, clearly one can arrange that $g_1$ or $g_3$ is surjective, and then take the direct sum.
Next one can find compatible maps $s_i:M_i\to F_i$ with $g_is_i=\id$.
Indeed, choose $s_1$, which induces $s_3$, and use that 
$F_2\to F_3\times_{M_3}M_2$ is surjective to get $s_2$.
This gives compatible isomorphisms $F_i\cong M_i\oplus\Ker(g_i)$,
so $M$ is a direct summand of $F$, and the assertion follows.
\end{proof}

\begin{Lemma}
\label{Le:IJ-Win}
Let $\pfd$ be a perfectoid ring.
For $\pfd_1$, $\pfd_2$, $\pfd_{12}$ as in Remark~\ref{Rk:IJ-perfd} 
the natural diagrams of window categories
\begin{equation}
\label{Eq:IJ-Win-Ainf}
\xymatrix@M+0.2em{
\Win(\u A{}_{\inf}(\pfd)) \ar[r] \ar[d] &
\Win(\u A{}_{\inf}(\pfd_2)) \ar[d] \\
\Win(\u A{}_{\inf}(\pfd_1)) \ar[r] &
\Win(\u A{}_{\inf}(\pfd_{12}))
}
\end{equation}
and
\begin{equation}
\label{Eq:IJ-Win-Acris}
\xymatrix@M+0.2em{
\Win(\u A{}_{\cris}(\pfd)) \ar[r] \ar[d] &
\Win(\u A{}_{\cris}(\pfd_2)) \ar[d] \\
\Win(\u A{}_{\cris}(\pfd_1)) \ar[r] &
\Win(\u A{}_{\cris}(\pfd_{12}))
}
\end{equation}
are $2$-cartesian.
\end{Lemma}

\begin{proof}
The rings $\pfd$, $\pfd_1$, $\pfd_2$, $\pfd_{12}$ form a cartesian diagram with surjective maps,
and the same holds for the associated rings $A_{\inf}$ and $A_{\cris}$,
the latter by Proposition~\ref{Pr:IJ-Acris}.
Thus the diagrams of frames that arise from \eqref{Eq:IJ-Win-Ainf} and \eqref{Eq:IJ-Win-Acris}
by deleting `$\Win$' are cartesian with surjective maps in all components.
Using Lemma~\ref{Le:LF-A12} the assertion follows easily.
\end{proof}

\begin{Prop}
\label{Pr:Win-inf-cris}
If $p\ge 3$, for every perfectoid ring $\pfd$ the functor 
\begin{equation}
\label{Eq:lambda*}
\lambda^*:\Win(\u A{}_{\inf}(\pfd))\to\Win(\u A{}_{\cris}(\pfd))
\end{equation}
associated to \eqref{Eq:lambda} is an equivalence of categories.
\end{Prop}

\begin{proof}
By Lemma~\ref{Le:IJ-Win} we can assume that $\pfd$ is either perfect or torsion free.
In the perfect case $\lambda$ is bijective.
Let $\pfd=W(S)/\xi$ where $S$ is perfect and $\xi$ is distinguished. 
If $\pfd$ is torsion free, $(p,\xi)$ is a regular sequence in $W(S)$,
and $\lambda^*$ is an equivalence by \cite[Proposition~2.3.1]{Cais-Lau}
(which requires $p\ge 3$).
\end{proof}

\subsection{Breuil-Kisin-Fargues modules}
\label{Se:BKF}

Let $\pfd=W(S)/\xi$ be perfectoid as before.
In the following we write $A_{\inf}=A_{\inf}(\pfd)=W(S)$.

\begin{Defn}
\label{Def:BK-perfd}
A (locally free) Breuil-Kisin-Fargues module for $\pfd$ is a pair $(\FM,\varphi)$ 
where $\FM$ is a finite projective $A_{\inf}$-module and where 
$\varphi:\FM^\sigma\to\FM$ is a linear map whose cokernel is annihilated by $\xi$.
We denote by $\BK(\pfd)$ the category of Breuil-Kisin-Fargues modules for $\pfd$.
\end{Defn}

In the case $\pfd=\OOO_K$ for a perfectoid field $K$, 
free $\varphi$-modules over $A_{\inf}$ are studied by Fargues \cite{Fargues:AuDela} 
in analogy with the classical theory of Breuil-Kisin modules \cite{Kisin:Crystalline},
and are called Breuil-Kisin-Fargues modules in \cite{BMS}.
Here we only consider $\varphi$-modules of height $1$, which correspond to $p$-divisible groups.
When $\pfd$ is a perfect ring, then $A=W(\pfd)$, 
and $\BK(\pfd)$ is the category of Dieudonn\'e modules over $\pfd$ in the usual sense.

\begin{Lemma}
\label{Le:DMB-proj}
For $(\FM,\varphi)\in\BK(\pfd)$ the $\pfd$-module $\Coker(\varphi)$ is  projective.
\end{Lemma}

\begin{proof}
Cf.\ Lemma~\ref{Le:DMA-proj}.
Let $\FN=\FM^\sigma$ and $\bar\FM=\FM\otimes_{A}\pfd$ and 
$\bar\FN=\FN\otimes_{A}\pfd$.
There is a unique linear map $\psi:\FM\to\FM^\sigma$ such that $\varphi\circ\psi=\xi$,
and we obtain an exact sequence of finite projective $\pfd$-modules
\begin{equation}
\label{Eq:BK-proj}
\bar\FN\xrightarrow{\;\bar\varphi\;}\bar\FM
\xrightarrow{\;\bar\psi\;}\bar\FN\xrightarrow{\;\bar\varphi\;}\bar\FM.
\end{equation}
We have to show that $\Image(\bar\psi)$ is a direct summand of $\bar\FN$.
This holds if and only if for each maximal ideal $\Fm\subset \pfd$ 
the base change of \eqref{Eq:BK-proj} to $k=\pfd/\Fm$ is exact.
We have $p\in\Fm$, so $k$ is a perfect field of characteristic $p$.
The natural homomorphism $A\to A_{\inf}(k)=W(k)$ maps $\xi$ to $p$.
Thus $\FM\otimes_{A}W(k)$ is a Dieudonn\'e module over $k$, 
and it follows that the base change of \eqref{Eq:BK-proj} under $\pfd\to k$ is exact as required. 
\end{proof}

Lemma~\ref{Le:DMB-proj} implies that there is an equivalence of categories
\begin{equation}
\label{Eq:Win-DMB}
\Win(\u A{}_{\inf}(\pfd))\to\BK(\pfd),
\end{equation}
given by $(M,\Filone M,F,F_1)\mapsto(\FM,\varphi)$ with $\FM=\Filone M$ and 
$\varphi(1\otimes x)=\xi F_1(x)$, see \cite[Lemma~2.1.15]{Cais-Lau}. 
The inverse functor is determined by $M=\FM^\sigma$ and 
$\Filone M=\{x\in M\mid \varphi(x)\in\xi\FM\}$ and $F(x)=1\otimes\varphi(x)$ for $x\in M$.

\begin{Remark}
\label{Rk:indep-xi}
The frame $\u A{}_{\inf}(\pfd)$ depends on the choice of $\xi$, but the functor
\[
\BK(\pfd)\to\Win(\u A{}_{\inf}(\pfd))\to\Win(\u A{}_{\cris}(\pfd))
\]
defined as the composition of  \eqref{Eq:lambda*} 
and the inverse of \eqref{Eq:Win-DMB} is independent of $\xi$ as is easily verified. 
\end{Remark}


\subsection{$p$-divisible groups over perfectoid rings}

Let $\pfd$ be a perfectoid ring. 
The functor $\Phi_{\u S}$ of \eqref{Eq:PhiS} for $\u S=\u A{}_{\cris}(\pfd)$ 
defined by evaluation of the crystalline Dieudonn\'e module is a functor
\[
\Phi_\pfd^{\cris}:\BT(\Spec \pfd)\to\Win(\u A{}_{\cris}(\pfd)).
\]

\begin{Prop}
\label{Pr:BT-WinB}
If $p\ge 3$ then the functor $\Phi_\pfd^{\cris}$ is an equivalence.
\end{Prop}

\begin{proof}
Since the ring $A_{\cris}(\pfd)=A_{\cris}(\pfd/p)$ is torison free by Proposition~\ref{Pr:IJ-Acris},
there is another frame 
\[
\u A{}_{\cris}(\pfd/p)=(A_{\cris}(\pfd/p),\Filone A_{\cris}(\pfd/p),\pfd/p,\sigma,\sigma_1)
\]
defined by $\Filone A_{\cris}(\pfd/p)=\Filone A_{\cris}(\pfd)+pA_{\cris}(\pfd)$ 
and $\sigma_1(x)=p^{-1}\sigma(x)$.
The identity is a strict frame homomorphism $j:\u A{}_{\cris}(\pfd)\to\u A{}_{\cris}(\pfd/p)$
over the projection $\pi:\pfd\to\pfd/p$, and we obtain a commutative diagram of functors
\[
\xymatrix@M+0.2em{
\BT(\Spec \pfd) \ar[r]^-{\Phi^{\cris}_{\pfd}} \ar[d]_\pi &
\Win(\u A{}_{\cris}(\pfd)) \ar[d]^j \\
\BT(\Spec \pfd/p) \ar[r]^-{\Phi^{\cris}_{\pfd/p}} &
\Win(\u A{}_{\cris}(\pfd/p))
}
\]
where $\Phi_{\pfd/p}^{\cris}$ is the functor $\Phi_{\u S}$ for $\u S=\u A{}_{\cris}(\pfd/p)$.
Here $\Phi_{\pfd/p}^{\cris}$ coincides with the functor of Theorem~\ref{Th:PhiR},
but this is not needed. 
Since the semiperfect ring $\pfd/p=S/\xi_0$ is balanced, 
the functor $\Phi_{\pfd/p}^{\cris}$ is an equivalence by Corollary~\ref{Co:BT-Acris-str};
see also Remark~\ref{Re:BT-Acris-str}.

For $G\in\BT(\Spec\pfd/p)$ and $\u M=\Phi_{\pfd/p}^{\cris}(G)$ there is a natural isomorphism
of $\pfd$-modules $M\otimes_{A_{\cris}(\pfd/p)}\pfd\cong\DD(G)_{\pfd}$. 
Since $p\ge 3$, the divided powers on the ideal $p\pfd$ are topologically nilpotent.
By the Grothendieck-Messing Theorem \cite{Messing:Crystals} 
and by \cite[Lemma~4.2]{Lau:Frames} it follows that lifts of $G$ under $\pi$ and lifts
of $\u M$ under $j$ correspond to lifts of the Hodge filtration in the same way.
Therefore the functor $\Phi_{\pfd}^{\cris}$ is an equivalence. 
\end{proof}

\begin{Thm}
\label{Th:BT-DM}
If $p\ge 3$, for every perfectoid ring $\pfd$ there is an equivalence 
\[
\BT(\Spec \pfd)\cong\BK(\pfd)
\]
between $p$-divisible groups and Breuil-Kisin-Fargues modules.
\end{Thm}

\begin{proof}
We have a chain of functors
\begin{equation}
\label{Eq:Th-BT-DM}
\BT(\Spec \pfd)\to\Win(\u A{}_{\cris}(\pfd))\leftarrow\Win(\u A{}_{\inf}(\pfd))\cong\BK(\pfd)
\end{equation}
where the last equivalence is \eqref{Eq:Win-DMB}.
For $p\ge 3$ the two arrows are equivalences by 
Propositions \ref{Pr:BT-WinB} and \ref{Pr:Win-inf-cris}.
\end{proof}

The equivalence of Theorem~\ref{Th:BT-DM} is independent
of the choice of the generator $\xi$ of the kernel of $A_{\inf}\to \pfd$; 
see Remark~\ref{Rk:indep-xi}. 


\section{Classification of finite group schemes}

The equivalence between $p$-divisible groups 
and Breuil-Kisin-Fargues modules over perfectoid rings
induces a similar equivalence for finite group schemes.
For a scheme $X$ let $\pGr(X)$ be the category of commutative 
finite locally free $p$-group schemes over $X$.

\subsection{A category of torsion modules}
\label{Se:Cat-tors-mod}

If $A$ is a $p$-adically complete and torsion free ring,
let $\T(A)$ be the category of finitely presented $A$-modules
of projective dimension $\le1$ which are annihilated by a power of $p$.

\begin{Lemma}
\label{Le:T(A)-bc}
For a homomorphism of $p$-adically complete torsion free rings $A\to A'$
and $M\in\T(A)$ we have $M\otimes_AA'\in\T(A')$.
\end{Lemma}

\begin{proof}
Let $0\to Q\xrightarrow uP\to M\to 0$ be exact where $P$ and $Q$ are finite projective $A$-modules.
Let $p^rM=0$.
There is a homomorphism $w:P\to Q$ such that $uw=p^r$ and $wu=p^r$.
Let $Q'=Q\otimes_AA'$ etc.
Since $Q'$ is torsion free it follows that $0\to Q'\to P'\to M'\to 0$ is exact,
thus $M'\in\T(A')$.
\end{proof}

The category $\T(A)$ can be described in terms of the rings $A/p^n$ as follows.

\begin{Lemma}
\label{Le:T(A)}
Let $A$ be a $p$-adically complete torsion free ring, $A_n=A/p^n$.
Let $M$ be a finite $A$-module annihilated by $p^r$.
We have $M\in\T(A)$ iff for every exact sequence $0\to Q_n\to P_n\to M\to 0$
where $P_n$ is a finite projective $A_n$-module with $n\ge r$, 
the $A_{n-r}$-module $Q_n/p^{n-r}Q_n$ is finite projective.
\end{Lemma}

\begin{proof}
Assume that $M\in\T(A)$ and let $0\to Q_n\to P_n\to M\to 0$ be as in the lemma.
Choose a finite projective $A$-module $P$ with $P/p^n=P_n$ and let $Q$ be
the kernel of $P\to M$. Then $Q$ is finite projective over $A$, and $Q_n=Q/p^nP$.
We have $p^nP\subseteq p^{n-r}Q$, 
and thus $Q_n/p^{n-r}=Q/p^{n-r}$ is finite projective over $A_{n-r}$.
Conversely, assume that the condition on $M$ holds and let $0\to Q\to P\to M\to 0$
be exact where $P$ is finite projective over $A$. 
For $n\ge r$ let $P_n=P/p^n$ and $Q_n=Q/p^nP$.
Then the $A_n$-module $\tilde Q_n=Q_{n+r}/p^nQ_{n+r}$ is finite projective, 
and we have $\tilde Q_{n+1}/p^n=\tilde Q_n$.
It follows that $Q=\varprojlim Q_n=\varprojlim\tilde Q_n$ is finite projective over $A$.
\end{proof}

The category $\T(A)$ satisfies fpqc descent in the following sense.

\begin{Lemma}
\label{Le:T(A)-descent}
Let $A\to A'$ be a homomorphism of $p$-adically complete torsion free rings
such that $A/p\to A'/p$ is faithfully flat. Let $A''$ and $A'''$ be the $p$-adic
completions of $A'\otimes_AA'$ and $A'\otimes_AA'\otimes_AA'$.
Then $\T(A)$ is equivalent to the category of pairs $(M',\alpha)$ where $M'\in\T(A')$
and $\alpha:M'\otimes_AA'\cong A'\otimes_AM'$ is an isomorphism that satisfies the
usual cocycle condition over $A'''$.
\end{Lemma}

\begin{proof}
Lemma~\ref{Le:T(A)-bc} gives a functor $M\mapsto(M',\alpha)$.
By the local flatness criterion $A/p^n\to A'/p^n$ is faithfully flat for each $n$.
It follows that the functor $M\mapsto(M',\alpha)$ is fully faithful, moreover each $(M',\alpha)$
with $M'\in\T(A)$ comes from an $A$-module $M$ annihilated by a power of $p$,
and it remains to show that $M\in\T(A)$. 
This is an easy consequence of Lemma~\ref{Le:T(A)}.
\end{proof}

The category $\T(A)$ preserves projective limits of nilpotent immersions:

\begin{Lemma}
\label{Le:T(A)-lim}
Let $A={\varprojlim}{}_n\, A^{n}$ for a surjective system $A^1\leftarrow A^2\leftarrow\ldots$
of $p$-adically complete torsion free rings such that $\Ker(A^{n+1}\to A^n)$ is nilpotent.
Then the obvious functor $\rho:\T(A)\to{\varprojlim}{}_n\,\T(A^n)$ is an equivalence.
\end{Lemma}

\begin{proof}
For $M\in\T(A)$ let $0\to Q\to P\to M\to 0$ be exact where $P$ and $Q$ are finite projective over $A$.
By the proof of Lemma~\ref{Le:T(A)-bc} the base change under $A\to A^n$ gives an exact sequence
$0\to Q^n\to P^n\to M^n\to 0$. Since $P={\varprojlim}{}_n\, P^n$ and $Q={\varprojlim}{}_n\, Q^n$ 
it follows that $M={\varprojlim}{}_n\, M^n$.
In particular the functor $\rho$ is fully faithful. 

Conversely, let $M^n\in\T(A^n)$ with isomorphisms $M^{n+1}\otimes_{A^{n+1}}A^n\cong M^n$ 
be given. Let $M={\varprojlim}{}_n\, M^n$ and choose a homomorphism $P\to M$ where
$P$ is finite projective over $A$ such that
$P^1\to M^1$ is surjective, $P^n=P\otimes_AA^n$. 
Then $P^n\to M^n$ is surjective by Nakayama's Lemma. 
The module $Q^n=\Ker(P^n\to M^n)$ is finite projective over $A^n$, 
and $Q^n=Q^{n+1}\otimes_{A^{n+1}}A^n$ by the proof of Lemma~\ref{Le:T(A)-bc}.
It follows that $Q={\varprojlim}{}_n\, Q^n$ is finite projective over $A$,
and $0\to Q\to P\to M\to 0$ is exact, thus $M\in\T(A)$.
The base change under $A\to A^n$ of the last sequence remains exact,
so $M^n=M\otimes_AA^n$.
\end{proof}

\subsection{Torsion Breuil-Kisin-Fargues modules}

Let $\pfd=W(S)/\xi$ be a perfectoid ring where $S$ is perfect and $\xi$
is distinguished.
We write again $A_{\inf}=A_{\inf}(\pfd)=W(S)$, 
and $\Filone A_{\inf}=\Ker(A_{\inf}\to \pfd)=\xi A_{\inf}$.

\begin{Defn}
A torsion Breuil-Kisin-Fargues module for $\pfd$ is a triple $(\FM,\varphi,\psi)$ 
where $\FM\in\T(A_{\inf})$ and where
\begin{equation}
\label{Eq:BKtor}
\Filone A_{\inf}\otimes_{A_{\inf}}\FM\xrightarrow{\;\psi\;} \FM^\sigma\xrightarrow{\;\varphi\;}\FM
\end{equation}
are linear maps such that $\varphi\circ\psi$ and $\psi\circ(1\otimes\varphi)$ 
are the multiplication maps. 
We denote by $\BK_{\tor}(\pfd)$ the category of torsion Breuil-Kisin-Fargues modules over $\pfd$.
\end{Defn}

\begin{Remark}
For a homomorphism of perfectoid rings $\pfd\to \pfd'$ there is an obvious base change functor
$\BK_{\tor}(\pfd)\to\BK_{\tor}(\pfd')$; see Lemma~\ref{Le:T(A)-bc}.
\end{Remark}

\begin{Remark}
If $\pfd$ is torsion free, $(p,\xi)$ is a regular sequence in $A_{\inf}$, 
thus $\xi$ is $\FM$-regular for each $\FM\in\T(A_{\inf})$,
and torsion Breuil-Kisin-Fargues are equivalent to pairs $(\FM,\varphi)$  
where the cokernel of $\varphi$ is annihilated by $\xi$.
\end{Remark}

\begin{Remark}
\label{Rk:BK-Coker}
For a locally free Breuil-Kisin-Fargues module
$\u\FM=(\FM,\varphi)$ as in Definition~\ref{Def:BK-perfd}
there is a unique $\psi$ as in \eqref{Eq:BKtor}.
In the following we will view $\u\FM$ as a triple $(\FM,\varphi,\psi)$.
A homomorphism $u:\u\FM\to\u\FM'$ in $\BK(\pfd)$ is called an isogeny if
it becomes bijective over $A[1/p]$. 
Then $u$ is injective, and its cokernel lies in $\BK_{\tor}(\pfd)$.
\end{Remark}

%

\subsubsection*{Etale descent}

By an abuse of notation, let $(\Spec \pfd/p)_{\et}$ 
denote the site of all affine etale $\pfd/p$-schemes,
with surjective families as coverings.
For an etale $\pfd/p$-algebra $B'$ there is a unique homomorphism of
perfectoid rings $\pfd\to \pfd'$ with $\pfd'/p=B'$; see Lemma~\ref{Le:tilt-etale}.
We define presheaves of rings $\AAA_{\inf}$ and $\RRR$ on $(\Spec\pfd/p)_{\et}$ by 
\[
\RRR(\Spec B')=\pfd',\qquad
\AAA_{\inf}(\Spec B')=A_{\inf}(\pfd').
\]
For varying etale $\pfd/p$-algebras $B'$, the categories $\LF(\pfd')$ 
of locally free $\pfd'$-modules 
form a fibered category $\LF(\RRR)$ over $(\Spec\pfd/p)_{\et}$.
Similarly we have fibered categories
$\LF(\AAA_{\inf})$, $\T(\AAA_{\inf})$,
$\BK(\RRR)$, $\BK_{\tor}(\RRR)$, $\BT(\Spec\RRR)$, and $\pGr(\Spec\RRR)$
over $(\Spec\pfd/p)_{\et}$;
see Lemma~\ref{Le:T(A)-bc} for $\T(\AAA_{\inf})$.

\begin{Lemma}
The presheaves of rings $\AAA_{\inf}$ and $\RRR$ on $(\Spec\pfd/p)_{\et}$ are sheaves.
The fibered categories $\LF(\RRR)$, $\LF(\AAA_{\inf})$, $\T(\AAA_{\inf})$,
$\BK(\RRR)$, $\BK_{\tor}(\RRR)$, $\BT(\Spec\RRR)$, and $\pGr(\Spec\RRR)$ 
over $(\Spec\pfd/p)_{\et}$ are stacks.
\end{Lemma}

\begin{proof}
Let $x=[\xi_0]\in A:=A_{\inf}$ and let $I=(x,p)$ as an ideal of $A$.
Then $A$ is $I$-adically complete. Let $B=R/p$.
We fix a faithfully flat etale homomorphism $B\to B'$
and write $A'=A_{\inf}(B')$ and $A''=A_{\inf}(B'\otimes_BB')$
and $A'''=A_{\inf}(B'\otimes_BB'\otimes_BB')$.
The reduction modulo $I^n$ of $A\to A'$ is etale,
and the reductions modulo $I^n$ of $A'\otimes_AA'\to A''$ and of 
$A'\otimes_AA'\otimes_AA'\to A'''$ are isomorphisms.
Since the category $\LF(A)$ is equivalent to ${\varprojlim}{}_n\,\LF(A/I^n)$, 
etale descent of locally free modules shows that $\AAA_{\inf}$ is a sheaf
and $\LF(\AAA_{\inf})$ and $\BK(\RRR)$ are stacks.
A similar argument shows that $\RRR$ is a sheaf and that $\LF(\RRR)$,
$\pGr(\Spec\RRR)$, and $\BT(\Spec\RRR)$ are stacks.

We claim that $A$ is $x$-adically complete
and that the quotients $A/x^n$ are $p$-adically complete and torsion free. 
Indeed, this is clear when $\pfd$ is perfect and thus $x=0$, or when $\pfd$ is torsion free; 
in that case $(x,p)$ is a regular sequence in $A$. 
In general, we use the exact sequence \eqref{Eq:IJ-perfd-WS} 
where $A=W(S)$. Let $A_i=W(S_i)$.
Since $x$ is zero in $A_2$ and in $A_{12}$ we get an exact sequence
$0\to A/x^n\to A_1/x^n\oplus A_2\to A_{12}\to 0$.
Here all rings except possibly $A/x^n$ are $p$-adically complete and torsion free, 
thus the same holds for $A/x^n$. 
The limit over $n$ shows that $A$ is $x$-adically complete. The claim is proved.

Lemma~\ref{Le:T(A)-lim} implies that $\T(A)$ is equivalent to ${\varprojlim}{}_n\T(A/x^n)$,
and similarly for $A'$ and $A''$ and $A'''$.
The homomorphism $A/(x^n,p)\to A'/(x^n,p)$ is faithfully flat etale, 
hence Lemma~\ref{Le:T(A)-descent} implies that
$\T(A/x^n)$ is equivalent to the category of modules in $\T(A'/x^n)$ with a
descent datum in $\T(A''/x^n)$. 
This proves that $\T(\AAA)$ and $\BK_{\tor}(\RRR)$ are stacks.
\end{proof}

Let us now continue the discussion of Remark \ref{Rk:BK-Coker}.

\begin{Lemma}
\label{Le:BKtor-quot}
For $\u\FM\in\BK_{\tor}(\pfd)$,
Zariski locally in $\Spec(\pfd/p)$ there is an isogeny of locally free Breuil-Kisin-Fargues
modules with cokernel\/ $\u\FM$.
\end{Lemma}

\begin{proof}
This is similar to \cite[Lemma~(2.3.4)]{Kisin:Crystalline}.
We have to find locally in $\Spec(\pfd/p)$ a surjective map $\u\FN\to\u\FM$ where $\u\FN$
is a locally free Breuil-Kisin-Fargues module.
One can choose finite free $A$-modules $Q$ and $\FN$ of equal rank 
and a commutative diagram with surjective vertical maps
\[
\xymatrix@M+0.2em{
\Filone A_{\inf}\otimes_{A_{\inf}}\FN \ar[r]^-g \ar[d]_{1\otimes\pi} & Q \ar[r]^f \ar[d]^\rho & \FN \ar[d]^\pi \\
\Filone A_{\inf}\otimes_{A_{\inf}}\FM \ar[r]^-{\psi} & \ar[r] \FM^\sigma \ar[r]^-\varphi & \FM 
}
\]
such that $f\circ g$ and $g\circ(1\otimes f)$ are the multiplication maps.
Assume that $u:\FN^\sigma\to Q$ is an isomorphism
with $\rho u=\sigma^*(\pi)$. 
Then $\u\FN=(\FN,fu,u^{-1}g)$ solves the problem.
It is easy to see that $u$ exists locally in $\Spec A$ and therefore also locally in $\Spec(\pfd/p)$.
\end{proof}

\begin{Lemma}
\label{Le:pGr-embed}
For $H\in\pGr(\Spec \pfd)$, Zariski locally in $\Spec(\pfd/p)$ there is an isogeny of $p$-divisible
groups with kernel $H$.
\end{Lemma}

\begin{proof}
We have to find locally in $\Spec(\pfd/p)$ an embedding of $H$ into a $p$-divisible group.
By \cite[Theorem~3.1.1]{BBM} such an embedding exists Zariski locally in $\Spec(\pfd)$,
and therefore also Zariski locally in $\Spec(\pfd/p)$.
\end{proof}

\begin{Thm}
\label{Th:pGr-DM}
If $p\ge 3$, for every perfectoid ring $\pfd$ there is an equivalence
\[
\pGr(\Spec \pfd)\cong\BK_{\tor}(\pfd).
\]
\end{Thm}

\begin{proof}
This follows from Theorem~\ref{Th:BT-DM} as in \cite[Th.\ (2.3.5)]{Kisin:Crystalline}.
More precisely, let $\pGr(\Spec \pfd)^\circ$ be the category of all $H\in\pGr(\Spec \pfd)$ 
with are the kernel of an isogeny in $\BT(\Spec \pfd)$,
and let $\BK_{\tor}(\pfd)^\circ$ be the category of all $\u\FM\in\BK_{\tor}(\pfd)$
which are the cokernel of an isogeny in $\BK(\pfd)$.
The corresponding fibered categories
$\pGr(\Spec\RRR)^\circ$ and $\BK_{\tor}(\RRR)^\circ$ over $(\Spec\pfd/p)_{\et}$ 
have associated stacks $\pGr(\Spec \RRR)$ and $\BK_{\tor}(\RRR)$ 
by Lemmas \ref{Le:BKtor-quot} and~\ref{Le:pGr-embed}.
Moreover $\pGr(\Spec \pfd)^\circ$ (resp.\ $\BK_{\tor}(\pfd)^\circ$)
is equivalent to the full subcategory of the derived category of the exact
category $\BT(\Spec \pfd)$ (resp.\ $\BK(\pfd)$) whose objects are isogenies
$G^0\to G^1$ (resp.\ isogenies $\u\FM_1\to\u\FM_0$).
Thus the equivalence of fibered categories $\BT(\Spec\RRR)\cong\BK(\RRR)$
given by Theorem~\ref{Th:BT-DM} induces an equivalence 
$\pGr(\Spec\RRR)\cong\BK_{\tor}(\RRR)$.
\end{proof}

\subsection{Torsion Dieudonn\'e modules}

For completeness we record a similar classification of finite group schemes
in the context of \S \ref{Se:DCM} and \S \ref{Se:Dmod-via-lifts}.

Let $R$ be an $\FF_p$-algebra and let $(A,\sigma)$ be a lift of $R$ as in \S \ref{Se:DCM}. 
A torsion Dieudonn\'e module over $A$ is a triple $\u M=(M,\varphi,\psi)$ where $M\in\T(A)$ 
and $\varphi:M^\sigma\to M$ and $\psi:M\to M^\sigma$ 
are linear maps with $\varphi\psi=p$ and $\psi\varphi=p$.
We write $\DM_{\tor}( A)$ for the category of torsion Dieudonn\'e modules over $A$.

An etale ring homomorphism $R\to R'$ extends to a unique homomorphism of lifts 
$(A,\sigma)\to (A',\sigma)$;
each $A/p^r\to A'/p^r$ is the unique etale homomorphism that lifts $R\to R'$.

\begin{Lemma}
\label{Le:PhiA-tor}
The functors $\Phi_{A'}:\BT(\Spec R')\to\DM(A')$ of \eqref{Eq:BT-DM} 
for all etale $R$-algebras $R'$ induce functors
\begin{equation}
\label{Eq:Fin-DM}
\Phi_{A'}^{\tor}:\pGr(\Spec R')\to\DM_{\tor}( A').
\end{equation}
If all functors $\Phi_{A'}$ are equivalences, then so are the functors $\Phi_{A'}^{\tor}$.
\end{Lemma}

\begin{proof}
The functor $\Phi_A$ induces a functor from the category of all
$H\in\pGr(\Spec R)$ which are the kernel of an isogeny of $p$-divisible groups
to the category of all $\FM\in\DM_{\tor}(A)$ which are the cokernel of an
isogeny of locally free Dieudonn\'e modules.
For given $H$ or $\FM$, such isogenies exist locally in $\Spec R$.
The lemma follows by descent; see Lemma~\ref{Le:T(A)-descent}.
\end{proof}

\begin{Cor}
\label{Co:PhiA-tor-equiv}
For a semiperfect ring $R$ with a lift $(A,\sigma)$ as in Theorem~\ref{Th:PhiA-equiv}, 
the functor $\Phi_A^{\tor}:\pGr(\Spec R)\to\DM_{\tor}(\u A)$ is an equivalence.
\end{Cor}

\begin{proof}
For each etale $R$-algebra $R'$ with associated lift $A'$ the resulting divided powers
on the kernel of $\phi:R'\to R'$ are induced from the divided powers on the kernel
of $\phi:R\to R$ and are thus pointwise nilpotent.
Hence $\Phi_{A'}$ is an equivalence by Theorem~\ref{Th:PhiA-equiv},
and Lemma~\ref{Le:PhiA-tor} applies.
\end{proof}


\end{document}